\documentclass[10pt]{article}
\usepackage[utf8]{inputenc}
\usepackage[top=30mm, left=30mm, right=30mm, bottom=30mm]{geometry}
\usepackage{amsmath,amssymb,amsthm}
\usepackage{dsfont}
\usepackage{color}
\usepackage{lineno,hyperref}
\usepackage{babel} 
\newtheorem{theorem}{Theorem}[section]
\newtheorem{lem}[theorem]{Lemma}

\newtheorem{kor}[theorem]{Corollary}
\newtheorem{prop}[theorem]{Proposition}
\theoremstyle{remark}
\newtheorem{rem}[theorem]{Remark}
\newtheorem{remark}[theorem]{Remark}
\theoremstyle{definition}
\newtheorem{dfn}[theorem]{Definition}

\numberwithin{equation}{section}

\newtheorem{claim}{Claim}

\def\dd{\displaystyle}
\def\rr{{\mathbb{R}}}
\def\rrd{{\mathbb{R}^d}}
\def\calp{{\mathcal{P}}}
\def\9{{\infty}}
\def\1{^{-1}}
\def\lbb{{\lambda}}
\def\vf{{\varphi}}
\def\pp{{\partial}}
\def\vp{{\varepsilon}}
\def\ff{\forall }

\def\barr{\begin{array}}
	\def\earr{\end{array}}
\def\({\left(}
\def\){\right)}
\def\FP{Fokker--Planck}

\newcommand{\esssup}{\mathrm{ess~sup}}
\DeclareMathOperator*{\supp}{supp}
\DeclareMathOperator*{\divv}{div}

\DeclareMathOperator*{\loc}{loc}

\newcommand*{\Ascr}{\mathcal A}
\newcommand*{\Bscr}{\mathcal B}

\newcommand*{\Fscr}{\mathcal F}

\newcommand*{\Mscr}{\mathcal M}

\newcommand*{\Pscr}{\mathcal P}

\newcommand*{\N}{\mathbb{N}}

\newcommand*{\R}{\mathbb{R}}

\newcommand*{\vrho}{\varrho}

\definecolor{orange}{rgb}{1.0, 0.55, 0.0} 

\newcommand{\footremember}[2]{%
	\footnote{#2}
	\newcounter{#1}
	\setcounter{#1}{\value{footnote}}%
}

\title{\bf $p$-Brownian motion and the $p$-Laplacian}
\author{Viorel Barbu\footnote{Al.I. Cuza University and Octav  Mayer Institute of Mathematics of Romanian Academy, Ia\c si, Romania. E-mail: vb41@uaic.ro} \and
	Marco Rehmeier\footremember{alley}{Faculty of Mathematics, Bielefeld University, 33615 Bielefeld, Germany. E-mail: mrehmeier@math.uni-bielefeld.de}
	\and Michael Röckner\footnote{Faculty of Mathematics, Bielefeld University, 33615 Bielefeld, Germany. E-mail: roeckner@math.uni-bielefeld.de} \footnote{Academy of Mathematics and System Sciences, CAS, Beijing}}
\date{\today}

\begin{document}
	\maketitle
	
	\begin{abstract}
		In this paper we construct a stochastic process, more precisely, a (nonlinear)  Markov process, which is related to the parabolic $p$-Laplace equation in the same way as Brownian motion is to the classical heat equation given by the (2-) Laplacian.  
	\end{abstract}
	
	
	\section{Introduction}
	
	As it has been known for more than 100 years  (where, e.g, the articles by A. Einstein \cite{17a}, M.~von Smoluchowski \cite{45a} and N.~Wiener~\cite{48a} are among the pioneering papers), there is a close relationship between Brownian motion and the Laplace operator, more precisely the classical heat equation, i.e.,
	\begin{equation}\label{e1.1}
	\frac\pp{\pp t}\,u(t,x)=\Delta u(t,x),\ (t,x)\in(0,\9)\times\rrd,\end{equation}
	where $\Delta={\rm div}\,\nabla$. 
	But it has been an open question whether the same is true for the $p$-Laplace operator, i.e. whether there exists a "$p$-Brownian motion" related to the parabolic $p$-Laplace equation, i.e., for $p>2$,
	\begin{equation}\label{e1.2}
	\frac\pp{\pp t}\,u(t,x)={\rm div}(|\nabla u(t,x)|^{p-2}\nabla u(t,x)),\ (t,x)\in(0,\9)\times\rrd,\end{equation}
	in an analogous way as in the case $p=2$, i.e. \eqref{e1.1}. The $p$-Laplacian 
	$$\Delta_pu={\rm div}(|\nabla u|^{p-2}\nabla u)$$
	has been for many years and still is an extensively studied operator in a large number of contexts, be it from the point of view of partial differential equations (PDEs) (see, e.g., \cite{13a}, \cite{14a}, 
	\cite{24a}, \cite{29a}, \cite{31a},
	\cite{32a},  \cite{44a} and the references therein), within nonlinear functional analysis, e.g. as the duality map between $W^{1,p}$ and its dual $W^{-1,p'}$ (see, e.g., \cite{10'}, \cite{31'}),   in nonlinear potential theory (see, e.g., \cite{1a},  \cite{23a}, \cite{33a} and the references therein) or its applications in physics (see, e.g., \cite{27'}, \cite{34'}). 
	
	It should be mentioned that linear potential theory has played a crucial role in developing and exploiting the relation between the (2-) Laplacian and Brownian motion and, more generally, between large classes of linear partial (and pseudo) differential operators and their associated Markov processes   for more than 60 years (see, e.g., \cite{7a}, \cite{8a},  \cite{15a}, \cite{16a}, \cite{18a},  \cite{21a}, \cite{22a}, \cite{25a}, \cite{30a}, \cite{38a}, \cite{39a}, \cite{40a} and \cite{41a}). Hence, one could expect that non-linear potential theory, which also has a long and exciting history, would pave the way for discovering the probabilistic counterpart of the $p$-Laplacian. One approach in this direction was developed in the fundamental papers  \cite{34''}, \cite{35'}, where a deep relation of a stochastic game, the "tug-of-war" game with noise, was discovered and exploited to find a beautiful probabilistic description of the $p$-harmonic function solving the Dirichlet problem for the $p$-Laplacian on a bounded domain in $\rrd$. In this paper we propose a different approach, namely to construct the desired probabilistic counterpart to the $p$-Laplacian as a Markov process which is related to the $p$-Laplace equation in the same way as Brownian motion is to the classical heat equation \eqref{e1.1}, and which then may be called a "$p$-Brownian motion". 
	
	The starting point in our approach is to recall that the relation between the classical heat equation and Brownian motion is a special case of the relation between a Fokker--Planck equation (FPE) (see, e.g., \cite{9a}) and its associated Markov process. In fact, this relation is a one-to-one correspondence, which turns out to extend to large classes of FPEs and their  canonically associated Markov processes, and even to nonlinear FPEs and their  canonically associated nonlinear Markov processes (see \cite{35a}, \cite{5a}, \cite{37a}). We will explain this below in more detail and only mention here that such nonlinear FPEs can be shown to have an associated family of stochastic processes given  as the solutions to a corresponding McKean--Vlasov stochastic differential equation (SDE), which turns out to constitute a nonlinear Markov process (though the latter is extremely hard to prove in concrete cases).   
	
	Key Step 1 in our construction is to identify the parabolic $p$-Laplace equation as a nonlinear FPE, and Key Step 2 to solve the corresponding  McKean--Vlasov SDE. The third and hardest Key Step is then to prove that the solutions of the latter constitute a nonlinear Markov process which is thus associated to the parabolic $p$-Laplace equation in the same way as Brownian motion is to the classical heat equation. 
	
Let us now explain in more detail the relation of nonlinear FPEs and nonlinear Markov processes, where the latter have already been defined in McKean's classical paper \cite{35a} (see also  \cite{37a}). Let $\calp(\rrd)$ denote the space of Borel probability measures on $\rrd$, and for $1\le i,j\le d$ consider measurable maps
	$$b_i,a_{ij}:[0,\9)\times\rrd\times\calp(\rrd)\to\rr$$such that the matrix $(a_{ij})_{i,j}$ is pointwise symmetric and nonnegative definite. Then, a nonlinear FPE is an equation of type 	
	\begin{equation}\label{e1.3}
	\frac\pp{\pp t}\,\mu_t=\frac\pp{\pp x_i}\,\frac\pp{\pp x_j}\,(a_{ij}(t,x,\mu_t)\mu_t)-\frac\pp{\pp x_i}\,(b_i(t,x,\mu_t)\mu_t),\ (t,x)\in(0,\9)\times\rrd,
	\end{equation}
	where the solution $[0,\9)\ni t\mapsto\mu_t$ is a curve in $\calp(\rrd)$ with some specified initial condition $\mu_0$. 
	Here, we use Einstein's  summation convention,  \eqref{e1.3} is meant as a weak equation in the sense of Schwartz distributions (see Definition \ref{dD.2} in Appendix D) and, henceforth, we call solutions to \eqref{e1.3} {\it distributional solutions}. The (in space) dual operator to the operator on the right hand side  of \eqref{e1.3} is called the corresponding Kolmogorov operator $L$, i.e. its action on test functions $\vf\in  C^\9_0((0,\9)\times\rrd)$ is given as
	\begin{equation}\label{e1.4}
	L\vf(t,x)=a_{ij}(t,x,\mu_t)\frac\pp{\pp x_i}\ \frac\pp{\pp x_j}\,\vf(t,x)+b_i(t,x,\mu_t)\frac\pp{\pp x_i}\,\vf(t,x),\ (t,x)\in(0,\9)\times\rrd.\end{equation}
	In turn, this operator determines the corresponding McKean--Vlasov SDE
	\begin{eqnarray}
	dX(t)&=&b(t,X(t),\mu_t)dt+\sqrt{2}\,\sigma(t,X(t),\mu_t)dW(t),\ t>0, 
	\label{e1.5}\\
	\mathcal{L}_{X(t)}&=&\mu_t,\ t\ge0,  \label{e1.6}
	\end{eqnarray}
	where $\sigma=(\sigma_{ij})_{ij}$ with $\sigma\sigma^\top=(a_{ij})_{ij}$, $b=(b_1,...,b_d)$, $W(t)$, $t\ge0$, is a (standard) $d$-dimensional Brownian motion on some probability space $(\Omega,\mathcal{F},\mathbb{P})$, and the maps $X(t):\Omega\to\rrd$, $t\ge0$, form the continuous in $t$ solution process to \eqref{e1.5} such that its one-dimensional time marginals
	\begin{equation}\label{e1.7} 
	\mathcal{L}_{X(t)}:=X(t)_*\mathbb{P},\ t\ge0,
	\end{equation}
	(i.e. the push forward or image measures of $\mathbb{P}$ under $X(t)$) satisfy \eqref{e1.6}. We refer to \cite{12a}, \cite{34a},  \cite{42a}, \cite{46a}  and the references therein for details on McKean--Vlasov SDEs. The correspondence between the nonlinear FPE \eqref{e1.3} and the McKean--Vlasov SDE \eqref{e1.5}, \eqref{e1.6} 
	is now given as follows. Suppose one has a solution to \eqref{e1.5}, \eqref{e1.6}, then by a simple application of It\^o's formula $\mu_t:=\mathcal{L}_{X(t)}$, $t\ge0$, solves \eqref{e1.3}. Conversely,  suppose one has a distributional solution to \eqref{e1.3}, then the nonlinear version in \cite{3a},   \cite{4a} of the superposition principle \cite[Theorem 2.5]{43a} (which in turn is a generalization of \cite{20a}, see also \cite{27a} and \cite{11a}) together with  \cite[Proposition 2.2.3]{22'} (which is a slight generalization of \cite[Proposition 4.11]{24'}) 
	 gives rise to a solution to \eqref{e1.5}, \eqref{e1.6}. We would like to emphasize that in \cite{4a} (see also \cite{5a})  
	apart from some integrability conditions the coefficients are only required to be measurable  (in contrast to the cases studied in \cite{12a}). 
	
	Obviously, the special case of classical Brownian motion and the classical heat equation is the case where $a_{ij}(t,x,\mu)=\delta_{ij}$ (= Kronecker delta),  $b_i(t,x,\mu)=0$, i.e. \eqref{e1.3} turns into
	\begin{equation}\label{e1.8}
	\frac\pp{\pp t}\,\mu_t=\Delta\mu_t,\ (t,x)\in(0,\9)\times\rrd,
	\end{equation}
	and \eqref{e1.5}, \eqref{e1.6} into  
	\begin{eqnarray}
	dX(t)&=&dW(t),\ t>0,\label{e1.9}\\
	\mathcal{L}_{X(t)}&=&\mathcal{L}_{W(t)}=\mu_t,\ t\ge0,\label{e1.10}
	\end{eqnarray}
	and, of course, $\mu_t$ is absolutely continuous with respect to  Lebesgue measure $dx$ on $\rrd$ with density $u(t,x)$, so \eqref{e1.8} is really \eqref{e1.1}. 
	
	To see that also \eqref{e1.2} is of type \eqref{e1.3}, we recall that the coefficients in \eqref{e1.3} only need to be measurable in the variable $\mu$, so in case the solutions $\mu_t$, $t\ge0$, have densities $u(t,\cdot)$, $t\ge0$, we might, e.g.,  have dependencies as follows
	\begin{equation} \label{e1.11}
	\barr{rcl}
	a_{ij}(t,x,\mu_t)&=&\tilde a_{ij}(t,x,\Gamma_1(u)(t,x)),\\
	b_i(t,x,\mu_t)&=&\tilde b_i(t,x,\Gamma_2(u)(t,x)),\earr 
	\end{equation}
	where $\tilde b_i,\tilde a_{ij}:[0,\9)\times\rrd\times\rr^k\to\rr$ and each $\Gamma_i$ is a functional on the space of distributional solutions whose values are again measurable functions of $t$ and $x$. Then, rewriting \eqref{e1.2} as 
	\begin{equation}\label{e1.13} 
	\frac\pp{\pp t}\,u(t,x)=\Delta(|\nabla u(t,x)|^{p-2}u(t,x))-{\rm div}(\nabla(|\nabla u(t,x)|^{p-2})u(t,x)),\ (t,x)\in(0,\9)\times\rrd,\end{equation}
	we see that \eqref{e1.2}, respectively \eqref{e1.13},  is of type \eqref{e1.3} with $a_{ij},b_i$ as in \eqref{e1.11}   with
	\begin{equation}\label{e1.13'}
	\barr{rcl}
		\tilde a_{ij}(t,x,\Gamma_1(u)(t,x))&=&\delta_{ij}|\nabla u(t,x)|^{p-2}\smallskip\\
		\tilde b(t,x,\Gamma_2(u)(t,x))&=&\nabla(|\nabla u(t,x)|^{p-2}).\earr
	\end{equation}
	We point out that while the already challenging so-called {\it Nemytskii-case}, where $\Gamma_i(u)(t,x)=u(t,x)$, has received more and more attention in the last years (see, for instance, \cite{2a}, \cite{3a}, \cite{4a}, \cite{5a} and the references therein), the coefficients in \eqref{e1.13'} even depend on $u$ via its first- and second-order derivatives. To the best of our knowledge, the relation of such nonlinear FPEs to McKean--Vlasov SDEs has not been studied before. 	
	The Kolmogorov operator associated with the coefficients \eqref{e1.13'} is then given by
	\begin{equation}\label{e1.14} 
	L\vf(t,x)=|\nabla u(t,x)|^{p-2}\Delta\vf(t,x)+\nabla(|\nabla u(t,x)|^{p-2})\cdot\nabla\vf(t,x),\ (t,x)\in(0,\9)\times\rrd,\end{equation}
	and the corresponding McKean--Vlasov SDE by
	\begin{eqnarray}
	dX(t)&=&\nabla(|\nabla u(t,X(t))|^{p-2})dt+\sqrt{2}\,|\nabla u(t,X(t))|^{\frac{p-2}2}dW(t),\ t>0,\label{e1.15} \\
	\mathcal{L}_{X(t)}&=&u(t,x)dx,\ t>0.\label{e1.16}
	\end{eqnarray} 
	This completes Key Steps 1 and 2. 
	
	For Key Step 3, i.e. to obtain the nonlinear Markov process from \eqref{e1.15}, \eqref{e1.16}, we need to take into account initial conditions which we choose to be Dirac measures $\delta_y$, $y\in\rrd$. So, we impose in \eqref{e1.13} that $u(0,x)dx=\delta_y$ and in \eqref{e1.16}  that $\mathcal{L}_{X(0)}=\delta_y$. In both cases $p=2$ and $p>2$, for such initial conditions the fundamental solution to \eqref{e1.1} and \eqref{e1.13} (that is, to \eqref{e1.2})  are explicitly known, namely in case $p=2$, for $y\in\rrd$ it is given by the classical Gaussian heat kernel
	\begin{equation}\label{e1.17}
	u^y(t,x):=\frac1{(4\pi t)^{\frac d2}}\exp\(-\frac1{4t}|x-y|^2\),\ (t,x)\in(0,\9)\times\rrd,\end{equation}
	and in case $p>2$ by the famous Barenblatt solution (see \cite{24a})
	\begin{equation}\label{e1.18}
	w^y(t,x)=t^{-k}\(C_1-qt^{-\frac{kp}{d(p-1)}}
	|x-y|^{\frac p{p-1}}\)^{\frac{p-1}{p-2}}_+,\ (t,x)\in(0,\9)\times\rrd,\end{equation}
	where  $k:=\(p-2+\frac pd\)\1,$ $q:=\frac{p-2}p\(\frac kd\)^{\frac1{p-1}}$,   $C_1\in(0,\9)$ is the unique constant such that $|w^y(t)|_{L^1(\rrd)}$ $=1$ for all $t>0$, and $f_+:=\max(f,0)$. Then, we consider   the {\it path laws} of the corresponding solutions $X^y(t)$, $t\ge0$, of \eqref{e1.9}, \eqref{e1.10} and \eqref{e1.15}, \eqref{e1.16}, respectively, namely,
\begin{equation}\label{e1.19}
	P_y:=(X^y)_*\mathbb{P},\ y\in\rrd,\end{equation}
	i.e. the push forward or image measure of $\mathbb{P}$ under the map $X^y:\Omega\to C([0,\9),\rrd)$ (= all continuous paths in $\rrd$).  Then, if $p=2$, $P_y$, $y\in\rrd$, form a Markov process in the sense of e.g. \cite{8a}, \cite{16a}, \cite{39a}, 
	also called {\it Brownian motion}, which is uniquely determined by $u^y(t,x),$ $y\in\rrd$, $t>0$. And  in this paper we prove that $P_y$, $y\in\rrd$, form a nonlinear Markov process in the sense of McKean \cite{35a}, which is uniquely determined by $w^y(t,x),$ $y\in\rrd$, $t>0$ and \eqref{e1.15}--\eqref{e1.16}. In contrast to the results of Sections \ref{s2}--\ref{s4}, which hold for all $d\ge1$, $p>2$, for this part, i.e. Key Step 3, we need $d\ge2$, $p>2$. By analogy, we call this nonlinear Markov process  {\it$p$-Brownian motion}. 
	
	The fact that the probability measures $P_y$, $y\in\rrd$, in \eqref{e1.19} form a nonlinear Markov process is of fundamental importance, since it pins down their interrelation and is the analogue on path space to the flow property of the fundamental solutions to the parabolic $p$-Laplace equation on state space (see \eqref{e5.1'} below). The proof of this fact is quite involved and one of its main ingredients is Theorem \ref{t5.10} below.  It uses very heavily the explicit form of the Barenblatt solution $w^y$ in \eqref{e1.18} and of the coefficients $|\nabla w^y(t,x)|^{p-2}$ and $\nabla(|\nabla w^y(t,x)|^{p-2})$ in \eqref{e1.13'} with $u=w^y$. This completes Key Step 3. 
	
	All these three Key Steps are implemented in detail in the main text of the paper, including stating or recalling all required definitions and notions.
	
	We would also like to make a remark concerning the above quoted references. The literature about the $p$-Laplacian, Markov processes, nonlinear potential theory, Fokker--Planck equations and McKean--Vlasov equations is so huge that it is impossible to give credit to it in an appropriate way in this paper. Therefore, above we tried to make a reasonably sized selection which is very probably not at all optimal. But this was necessary, in regard to the length of this paper. Nevertheless, we would like to comment on one more reference, namely 
	 \cite{26a}, 
	 to avoid confusion. We point out that the notion of nonlinear Markov processes therein is different from McKean's, and hence from ours in this paper, since it is defined there as 
	 a family of linear Markov processes. 
	 We refer to \cite{26a} for details.
	
	We firmly hope that the relation between the $p$-Brownian motion and the $p$-Laplacian, or more generally nonlinear  Markov processes and nonlinear FPEs,  will turn out to be as fruitful as in the linear theory, which profited a lot from transferring an analytic  problem for the PDE into a probabilistic problem for the stochastic process and vice versa. 
	
	Furthermore, we would like to mention that for the heat equation or more general linear FPEs,   considering Dirac initial conditions is in a sense generic, since by linearity in the initial condition a solution with another initial probability measure is a simple linear superposition of the solutions with Dirac initial conditions. This is completely different for the $p$-Laplacian equation, which by its nonlinearity produces a richer structure of solutions, if one considers general initial probability measures. Nevertheless, our general approach  above can be also applied to more general initial con\-di\-tions, and the corresponding McKean--Vlasov SDEs can be solved as above. This will be the subject of a forthcoming paper.
	
	Finally, it is obvious that Key Steps 1 and 2 above apply to much more general nonlinear parabolic PDEs than the parabolic $p$-Laplace equation \eqref{e1.2}. The hard part is Key Step 3, more precisely the extremality of the solution to the nonlinear PDE in the class of distributional solutions to the linearized PDE, analogous to Theorem \ref{t5.10} (see also Remark \ref{r6.6} (ii)) below. Also this will be the subject of our future work.
	
	Let us now summarize  the structure of the paper. While we have already presented  parts of Key Step 1, namely the identification of the parabolic $p$-Laplace equation as the nonlinear FPE \eqref{e1.13}   above, this step  is completed in Section \ref{s2}, where we define our notion of distributional solutions to \eqref{e1.13}   (see Definition \ref{d2.3}) and compare it with the usual notion of (weak) solution to the parabolic $p$-Laplace equation (see Definition \ref{d2.1}). Key Step 2 is implemented in Sections \ref{s3} and \ref{s4}. In Section \ref{s3}, we solve the corresponding McKean--Vlasov equation \eqref{e1.15}, \eqref{e1.16}   (see Definition \ref{d3.1} and Theorem \ref{t3.3}), and Section \ref{s4} is devoted to the study of the solutions to the nonlinear FPE \eqref{e1.13}   and to the construction of solutions to \eqref{e1.15}, \eqref{e1.16}   imposing Dirac initial conditions (see Proposition \ref{p4.1} and Theorem \ref{t4.2}). 	
	 Key Step 3 is presented in Sections \ref{s5} and \ref{s6a}. 
	  In Section \ref{s5} we recall McKean's definition  of a nonlinear Markov process (see Definition \ref{d5.1}) and then prove that the path laws $P_y$, $y\in\rrd$, (see \eqref{e1.19}) of the solutions to \eqref{e1.15}, \eqref{e1.16}   form such a process (see Theorem \ref{t5.2}) and we give the definition of the $p$-Brownian motion (see Definition \ref{d5.5}). The crucial ingredient already mentioned above, namely, the restricted distributional  uniqueness result for the linearized parabolic $p$-Laplace equation (see Theorem \ref{t5.10}) is subsequently proved in Section \ref{s6a}. The Appendix contains some specified details of proofs in the main text (see Appendices A--C) and some basic facts on FPEs (see Appendix D).

	\subsection{Notation}
	\textbf{Standard notation.} On $\R^d$ we write $|\cdot|$, $ x\cdot y$ and $dx$ (or $dt$ in one dimension) for the Euclidean norm, its inner product and Lebesgue measure. $\delta_{ij}$ denotes the Kronecker delta, and $\supp f$ the support of a function $f$ on a topological space.
	\\
	\textbf{Measures.}
	$\Mscr^+_b$ is the space of finite nonnegative Borel measures on $\R^d$. For a topological space $X$, we write $\Pscr(X)$ for the probability measures on $\Bscr(X)$, the Borel $\sigma$-algebra of $X$. For $X= \R^d$, we write $\Pscr(\R^d) = \Pscr$. A curve $t\mapsto \mu_t \in \Mscr^+_b$ on an interval $I \subseteq (-\infty,\infty)$ is \emph{vaguely}, respectively \emph{weakly continuous}, if $t\mapsto \int_{\R^d} \psi \, d \mu_t$ is continuous for all continuous compactly supported, respectively continuous bounded $\psi: \R^d \to \R$.
	The symbol $\mathbb{E}$ denotes expectation with respect to a given probability measure on a prescribed measurable space.
	\\
	\textbf{Function spaces.}
	For $U \subseteq \R^d$, a measure $\mu$ on $\Bscr(U)$ and a normed space $(X,|\cdot|_X)$, the spaces $L^p(U;X;\mu)$, $p\in [1,\infty]$, denote the usual spaces of $p$-integrable measurable functions $f: U \to X$ with norm $|f|^p_{L^p(U;X;\mu)} = \int_{U} |f|^p_{X}d\mu$ for $p< \infty$ and $|f|_{L^\infty(U;X;\mu)} = \inf\{C >0: |f|_X \leq C \,\mu-a.s.\}$. When $X =\R^1$ or $\mu = dx$, we omit $X$ or $dx$ from the notation. The corresponding local spaces are $L^p_{\loc}(U;X;\mu)$.
	
	For $m \in \N$ and $p \in [1,\infty]$, $W_{(\loc)}^{m,p}(\R^d;\R^k)$ are the usual Sobolev spaces of functions $g: \R^d \to \R^k$ with (locally) $p$-integrable weak partial derivatives up to order $m$. If $k=1$, we write $W^{m,p}_{(\loc)}(\R^d)$. For $p=2$, $H^m(\R^d,\R^k) := W^{m,2}(\R^d,\R^k)$ are Hilbert spaces with their usual norm $|\cdot|_{H^m}$ and dual space $H^{-m}(\R^d,\R^k)$. For an open interval $I \subseteq (-\infty,+\infty)$, $W^{2,1}_2(I\times \R^d)$ denotes the space of functions $g: I\times \R^d \to \R$ such that $g$ and its weak partial derivatives $\frac{d}{dt}g, \frac{\partial}{\partial x_i}g$ and $\frac{\partial}{\partial x_i} \frac{\partial}{\partial x_j}g$, $1\leq i,j \leq d$, belong to $L^2(I\times \R^d)$.
	
	For $U \subseteq \R^d$, a Banach space $X$ and $m \in \N \cup \{0\}$, $C^m(U,X)$ ($C^m_b(U,X)$, $C^m_0(U,X)$) are the spaces of continuous (bounded, compactly supported, respectively) functions $g: U \to X$ with continuous partial derivatives up to order $m$, written $C(U,X)$ ($C_b(U,X)$, $C_0(U,X)$) if $m=0$. If $X = \R^1$, we write $C^m(U), C^m_b(U)$ and $C^m_0(U)$, and we set $C^\infty(U) = \bigcap_{m >0} C^m(U)$ (likewise for $C^\infty_b(U)$ and $C^\infty_0(U)$). On $C([0,\infty),\R^d)$, $\pi_t$ denotes the projection $\pi_t(f):= f(t)$. In time-dependent situations, we denote by $\pi^s_t$ , $t\geq s$, the same map on $C([s,\infty),\R^d)$ and write $\pi^0_t = \pi_t$. For an interval $I \subseteq (-\infty,+\infty)$, $C^{2,1}_0(I\times \R^d)$ is the space of continuous functions $g: I \times \R^d \to \R$ with compact support in $I \times \R^d$ such that the pointwise derivatives $\frac{d}{dt}g$, $\frac{\partial}{\partial x_i}g$ and $\frac{\partial}{\partial x_i}\frac{\partial}{\partial x_j} g$, $1\leq i,j \leq d$, are continuous on $I\times\rrd$.

	\section{The $p$-Laplace equation and its formulation as a Fokker--Planck equation}\label{s2}
	
Regarding our Key Step 1 mentioned in the introduction, we begin by recalling the usual definition of solution to \eqref{e1.2} and to its \FP\ formulation \eqref{e1.13}. Let $p>2$.
	\begin{dfn}\label{d2.1}
		$u\in L^1_{\textup{loc}}((0,\infty); W^{1,1}_{\loc}(\R^d))$ is a \emph{solution} to \eqref{e1.2}, if 
		$\nabla u \in  L^{p-1}_{\textup{loc}}((0,\infty);L^{p-1}_{\textup{loc}}(\R^d;\R^d))$
		and for all $\vf \in C^1_0((0,\infty)\times \R^d)$
		\begin{equation}\label{e2.1}
		\int_0^\infty \int_{\R^d} -u\, \frac d{dt}\,\vf + |\nabla u|^{p-2}\nabla u \cdot \nabla \vf \,dx dt = 0.
		\end{equation}
		$u$ \emph{has initial condition} $\nu \in \Mscr_b^+$, if for every $\psi \in C_0(\R^d)$ there is a set $J_\psi \subseteq (0,\infty)$ of full $dt$-measure such that 
		\begin{equation}\label{e2.2}
		\int_{\R^d}\psi\, d\nu = \lim_{t \to 0, t \in J_\varphi}\int_{\R^d} u(t)\psi\, dx.
		\end{equation}
		A solution $u$ with initial condition $\nu$ is called \emph{probability solution}, if $\nu$ and $dt$-a.e. $u(t,x)dx$ belong to~$\calp$. 
\end{dfn}
		
		Condition \eqref{e2.2} for the initial condition is similar to \cite[(6.1.3)]{9a}. Note that in both parts of the following remark one needs the local integrability of $t\mapsto\nabla u(t)\in L^{p-1}_{(\rm loc)}(\rrd;\rrd)$ on $[0,\9)$, whereas in Definition \ref{d2.1} it is only required on $(0,\9)$.

	\begin{rem}\label{r2.2}\
		\begin{enumerate}
			\item [(i)] If $u$ is a solution with initial condition $\nu$ such that $t \mapsto u(t,x)dx$ is vaguely continuous on $(0,\infty)$, then $J_\psi = (0,\infty)$ for all $\psi\in C_0(\R^d)$ and $t \mapsto u(t,x)dx$ extends vaguely continuous in $t=0$ with value $\nu$. In this case, if also $\nabla u \in L^{p-1}_{\textup{loc}}([0,\infty);L^{p-1}_{\textup{loc}}(\R^d;\R^d))$, \eqref{e2.1}+\eqref{e2.2} are equivalent to
			\begin{equation}\label{e2.3}
			\int_{\R^d}\psi \,u(t)\,dx = \int_{\R^d} \psi \,d\nu  - \int_0^t \int_{\R^d} |\nabla u|^{p-2}\nabla u \cdot \nabla \psi \,dx dt,\quad\forall t \geq 0,
			\end{equation}
			for every $\psi \in C^1_0(\R^d)$. This can be proven as in the proof of \cite[Prop.6.1.2]{9a}.
			\item[(ii)] Every nonnegative solution $u$ with initial condition $\nu \in \Mscr^+_b$, $\esssup_{t>0}|u(t)|_{L^1(\rrd)} <\infty$ and $\nabla u \in L^{p-1}_{\textup{loc}}([0,\infty);L^{p-1}_{\textup{loc}}(\R^d;\R^d))$ has a unique vaguely continuous $dt$-version $\tilde{u}$ on $[0,\infty)$. This follows from \cite[Lemma 2.3]{36a}.  
			 Moreover, if $\nabla u \in L^{p-1}_{\textup{loc}}([0,\infty),L^{p-1}(\R^d))$, then $|u(t)|_{L^1(\rrd)} = \nu(\R^d)$ for all $t \geq 0$. The latter follows by considering \eqref{e2.3} for an increasing sequence $(\psi_n)_{n\in \N}\subseteq C^1_0(\R^d)$ such that $0\leq \psi_n \leq 1$, $\psi_n(x) = 1$ for $|x| < n$, $\sup_n|\nabla \psi_n(x)|\leq C$ for all $x\in \R^d$ for $C$ not depending on $n$ or $x$ and by letting $n \to \infty$. In this case, $t\mapsto\tilde{u}(t,x)dx$ is weakly continuous.
		\end{enumerate}
	\end{rem}
	
Now, we turn to the nonlinear \FP\ formulation \eqref{e1.13} of \eqref{e1.2} and its notion of distributional solution.

	\begin{dfn}\label{d2.3}
		$u\in L^1_{\textup{loc}}(0,\infty; W^{1,1}_{\loc}(\R^d))$ is a  \emph{distributional solution} to \eqref{e1.13}, if
		\begin{equation}\label{e2.4}
		\barr{l}
		|\nabla u|^{p-2} \in L^1_{\textup{loc}}((0,\infty);W^{1,1}_{\textup{loc}}(\R^d)),\quad	|\nabla u|^{p-2} u \in L^1_{\textup{loc}}((0,\infty)\times \R^d),\medskip\\ \nabla( |\nabla u|^{p-2}) u  \in L^1_{\loc}((0,\infty)\times \R^d;\R^d),\earr
		\end{equation}
		and for all $\vf \in C^2_0((0,\infty)\times \R^d)$
		\begin{equation}\label{e2.5}
		\int_0^\infty \int_{\R^d} u\(\frac d{dt}\, \vf +  |\nabla u|^{p-2} \Delta \vf+ \nabla( |\nabla u|^{p-2}) \cdot \nabla \vf\)\,dx dt = 0.
		\end{equation}
		$u$ \emph{has initial condition} $\nu \in \Mscr_b^+$, if for every $\psi \in C_0(\R^d)$ there is a set $J_\psi$ as in Definition \ref{d2.1} such that \eqref{e2.2} holds. A distributional solution $u$ with initial condition $\nu$ is a \emph{probability solution}, if $\nu$ and $dt$-a.e. $u(t,x)dx$ belong to $\mathcal{P}$.
	\end{dfn}

	\begin{rem}\label{r2.4}
		Remark \ref{r2.2} (i) holds accordingly: if $u$ is a distributional solution to \eqref{e1.13}  in the sense of Definition \ref{d2.3} with initial condition $\nu$ such that $t \mapsto u(t,x)dx$ is vaguely continuous, then $J_\varphi = \emptyset$ and, in this case, if the second and third condition in \eqref{e2.4} hold with $L^1_{\textup{loc}}([0,\infty)\times \R^d)$ and $L^1_{\loc}([0,\infty)\times \R^d; \R^d)$, replacing $L^1_{\rm loc}((0,\9)\times\rrd)$ and $L^1_{\rm loc}((0,\9)\times\rrd;\rrd)$,  respectively, then \eqref{e2.5} and $u$ having initial condition $\nu$ are equivalent to 
		\begin{equation}\label{e2.6}
		\int_{\R^d} \psi \,u(t)\,dx = \int_{\R^d} \psi \,d\nu + \int_0^t \int_{\R^d} \big( |\nabla u|^{p-2} \Delta \psi+ \nabla( |\nabla u|^{p-2}) \cdot \nabla \psi\big)u\,dx dt,\  \forall t \geq 0
		\end{equation}
		for every $\psi \in C^2_0(\R^d)$.
	\end{rem}
	
	\begin{lem}\label{l2.5}
		Assume $u\in L^1_{\textup{loc}}((0,\infty); W^{1,1}_{\loc}(\R^d))$ satisfies $\nabla u \in L^{p-1}_{\textup{loc}}((0,\infty);L^{p-1}_{\textup{loc}}(\R^d;\R^d))$ and \eqref{e2.4}. Then $u$ satisfies Definition {\rm\ref{d2.1}} with initial condition $\nu$ if and only if $u$ satisfies Definition {\rm\ref{d2.3}} with initial condition $\nu$. 
		
		If for $\mu_0:=\nu$, $\mu_t(dx):=u(t,x)dx,$ $t>0$, in addition, we have that  $[0,\9)\ni t\mapsto \mu_t$ is vaguely continuous and the above local in time integrability conditions hold on $[0,\infty)$ instead of $(0,\infty)$, then $u$ satisfies \eqref{e2.3} for all $\psi \in C^1_0(\R^d)$ if and only if it satisfies \eqref{e2.6} for all $\psi \in C^2_0(\R^d)$.
	\end{lem}

	\begin{proof}
		The second assertion follows from the first, since by Remarks \ref{r2.2} and \ref{r2.4}, under the additional assumption, \eqref{e2.3} and \eqref{e2.6} are equivalent to Definitions \ref{d2.1} and \ref{d2.3}, respectively. The first assertion follows, since under the assumption, the integrals on the left hand side  of \eqref{e2.1} and \eqref{e2.5} are equal for each $\vf \in C^2_0((0,\infty)\times \R^d)$, as can be seen by integration by parts. 	
	\end{proof}	
	\begin{remark}\label{r2.6} We point out that the fundamental solutions $w^y$ from \eqref{e1.18} solve both \eqref{e1.2} and \eqref{e1.13} in the sense of Definitions \ref{d2.1} and \ref{d2.3}, respectively, and  all additional assumptions from  Remarks \ref{r2.2} and \ref{r2.4} are satisfied. Details are given in Section \ref{s4}. 
		\end{remark}
		
	\section{Corresponding  McKean--Vlasov SDE and its solution}\label{s3}
	
	We now turn to Key Step 2 from the introduction, i.e. we associate with \eqref{e1.2} (more precisely, with \eqref{e1.13}) the McKean--Vlasov SDE \eqref{e1.15}--\eqref{e1.16}. Equation \eqref{e1.13}  is a nonlinear Fokker--Planck equation with coefficients as in \eqref{e1.13'}. These coefficients are 	defined for those measures $\mu$ in $ \Mscr^+$ such that $\mu = u(x)dx$ with $u, \,|\nabla u |^{p-2} \in W^{1,1}_{\textup{loc}}(\R^d)$. 
	
\begin{dfn}\label{d3.1}\
		\begin{enumerate}
			\item [(i)] 	A \emph{probabilistically weak solution} (short: \textit{solution}) to \eqref{e1.15}--\eqref{e1.16} is a triple consisting of a filtered probability space $(\Omega, \Fscr, (\Fscr_t)_{t\geq 0}, \mathbb{P})$, an $(\Fscr_t)$-adapted $\R$-valued stochastic process $X=(X(t))_{t\geq 0}$ and an $(\Fscr_t)$-Brownian motion $W$ such that 
			\begin{equation*}\barr{c}
			\dd	\mathbb{E}[\int_0^T | \nabla(|\nabla u(t,X(t))|^{p-2})| +  |\nabla  u(t,X(t))|^{p-2} dt] < \infty, \ \forall T>0,\smallskip\\
			\mathcal{L}_{X(t)}=u(t,x)dx,\ \ff t>0,\earr
			\end{equation*}
			and $\mathbb{P}$-a.s.
			\begin{equation*}
			X(t) = X_0 + \int_0^t   \nabla(|\nabla u(s,X(s))|^{p-2})ds+  \sqrt{2}\int_0^t  |\nabla  u(s,X(s))|^{\frac{p-2}{2}}  dW(s), \quad\forall  t \geq 0.
			\end{equation*}
			We say \textit{$X$ has initial condition} $\nu \in \mathcal{P}$, if $\mathcal{L}_{X(0)} = \nu$. 	Often, we shortly refer to $X$ as the solution.
			\item[(ii)] We call the path law $X_*\mathbb{P} \in \mathcal{P}(C([0,\infty),\R^d)$ of a solution to \eqref{e1.15}--\eqref{e1.16} a \textit{solution path law} to \eqref{e1.15}--\eqref{e1.16}. We say \textit{$P$ is the unique solution path law} ({\it with initial condition}  
			$\nu \in \mathcal{P}=\mathcal{P}(\rrd)$), 
			if $X_*\mathbb{P} = P$ for all weak solutions $X$ with  initial condition $\nu$. 
		\end{enumerate}
	\end{dfn}

	\begin{remark}\label{r3.2}\
		\begin{itemize}
	\item[(i)] 	 It is obvious how to generalize the previous definition to initial times $s>0$. In this case, the initial condition is the pair $(s,\nu) \in [0,\infty)\times \mathcal{P}$, and the solution is defined  on $[s,\9)$.
	
	\item[(ii)] As part of the definition it is assumed that for $dt$-a.e. $t >0$ the measure $\mathcal{L}_{X(t)}\in\calp$ has a Radon--Nikodym density $u(t,x)$ with respect to $dx$ with sufficient Sobolev regularity in order to make sense of the appearing integrands, an aspect we explicitly deal with for the fundamental solutions from  \eqref{e1.18} in Section \ref{s4}. 
\end{itemize}
\end{remark}

By a fundamental result by Trevisan \cite[Theorem 2.5]{43a}, known as {\it superposition principle}, in conjunction with \cite[Proposition 2.2.3]{22'} and \cite[Proposition 4.11]{24'}, for any weakly continuous (in general measure-valued) probability solution $(\mu_t)_{t\ge0},\mu_t\in\calp$,  to a linear \FP\ equation with measurable coefficients $a_{ij},b_i:(0,\9)\times\rrd\to\rr$, $1\le i,j\le d$, in the sense of Definition \ref{dD.1} there exists a (probabilistically weak) solution $X=(X(t))_{t\ge0}$ to the SDE 
$$dX(t)=b(t,X(t))dt+\sqrt{2}\,\sigma(t,X(t))dW(t),\ t\ge0,$$
where $b=(b_i)_{i\le d},$ $\sigma=(\sigma_{ij})_{i,j\le d}$ with $(\sigma\sigma^\top)_{ij}=a_{ij}$, and $W$ a $d$-dimensional Brownian motion, with one-dimensional time marginals  $\mathcal{L}_{X(t)}=\mu_t,$ ${t\ge0}$. 

In order to associate with solutions to \eqref{e1.2} (more precisely, \eqref{e1.13}) solutions to \eqref{e1.15}--\eqref{e1.16}, we follow the approach developed in \cite{3a},   \cite{4a}, i.e. we use this superposition principle in the nonlinear case by first fixing a probability solution $u$ to \eqref{e1.13} in the sense of Definition \ref{d2.3} in the coefficients \eqref{e1.13'} to obtain linear (that is, $(t,x)$-depending) coefficients. Then, we apply Trevisan's result to the same solution $u$ for the resulting linear equation. Clearly, the process $X=(X(t))_{t\ge0}$ obtained this way solves \eqref{e1.15}--\eqref{e1.16}.  Thus, we obtain the following new result.

\begin{theorem}\label{t3.3}
		Let $u$ be a weakly continuous probability solution to \eqref{e1.2} in the sense of Definition {\rm\ref{d2.1}} such that $$|\nabla u|^{p-2} \in L^1_{\textup{loc}}((0,\infty);W^{1,1}_{\textup{loc}}(\R^d)).$$ If 
		\begin{equation}\label{e4.1}
		\int_0^T \int_{\R^d}\big( |\nabla u|^{p-2} + |\nabla (|\nabla u|^{p-2})|\big)\, u\, dx dt < \infty, \quad \forall T>0,
		\end{equation}
		then there exists a   solution $X=(X(t))_{t\geq 0}$ to the McKean--Vlasov 
		SDE \eqref{e1.15}--\eqref{e1.16} in the sense of Definition {\rm\ref{d3.1}} such that $u(t,x)dx = \mathcal{L}_{X(t)}$ for all $t \geq 0$.
	\end{theorem}
	\begin{proof}
		By \eqref{e4.1} and Lemma \ref{l2.5}, $u$ solves equation \eqref{e1.13}  in the sense of Remark \eqref{r2.4}. Thus, $u(t,x)dx$ solves the linear \FP\ equation with coefficients \begin{equation}\label{e3.2'}
			a_{ij}(t,x)=\delta_{ij}|\nabla u(t,x)|^{p-2},\ b_i(t,x)=\frac\pp{\pp x_i}(|\nabla u(t,x)|^{p-2}),\ 1\le i,j\le d,\end{equation}in the sense of Definition {\rm\ref{dD.1}}. By \cite[Theorem 2.5]{43a}, there is a 
		 solution $X=(X(t))_{t\ge0}$ to the cor\-res\-pon\-ding (non-distribution dependent) SDE with coefficients $b$ and $\sqrt{2}\,\sigma$, where $b=(b_1,...,b_d)$ is as in \eqref{e3.2'} and $(\sigma\sigma^\top)_{ij}=a_{ij}$, with $a_{ij}$ as in \eqref{e3.2'},  such that $$\mathcal{L}_{X(t)}=u(t,x)dx,\ t\ge0.$$		
		Clearly, $X=(X(t))_{t\ge0}$  solves the McKean--Vlasov SDE \eqref{e1.15}--\eqref{e1.16}.
		\end{proof}

By Theorem \ref{t3.3}, solutions to the $p$-Laplace equation \eqref{e1.2} are represented as the one-dimensional time marginal density curves of stochastic processes, which solve \eqref{e1.15}--\eqref{e1.16}.

\section{Dirac initial conditions and fundamental solutions}\label{s4}

Now we implement Key Step 2 specifically for the fundamental solutions $w^y$ from \eqref{e1.18}. As said before, $w^y,$ $y\in\rrd$, solves both \eqref{e1.2} and \eqref{e1.13} in the sense of Definition \ref{d2.1} and \ref{d2.3}, respectively. This follows from part (i) of the following proposition and from the proof of Theorem \ref{t4.2} together  with Lemma \ref{l2.5}, respectively. 	
We collect some properties of $w^y$ in the following proposition. 

	\begin{prop}\label{p4.1} Let $y\in\rrd$. 
		\begin{enumerate}
			\item [\rm(i)] 	$w^y$ is a weakly continuous probability solution to \eqref{e1.2} with initial condition $\nu = \delta_y$ in the sense of Definition {\rm\ref{d2.1}}.
			\item[\rm(ii)] $w^y(t)$ is radially symmetric and compactly supported for every $t>0$. More precisely $\supp w^y(t)$ is contained in the ball centered at $y$ with radius $\(\frac{C_1t^\frac{kp}{d(p-1)}}{q}\)^{\frac{p-1}{p}}$.
			\item[\rm(iii)] $w^y$ is the unique nonnegative vaguely continuous solution to \eqref{e1.2} in the sense of Definition \ref{d2.1} with initial condition $\delta_y$.
		\end{enumerate}
	\end{prop}
We note that $w^y$ has the expected parabolic regularity, i.e. while the initial condition is a degenerate measure, at each time $t>0$, $w^y(t)$ is a continuous, weakly differentiable, compactly supported function on $\rrd$. 

\begin{proof}  (i) and (ii) can directly be inferred  from the definition of $w^y$. For (iii), we refer to \cite{24a}.
	\end{proof}
Next, we show that Theorem \ref{t3.3} applies to $w^y$.

	\begin{theorem}\label{t4.2}
		Let $d\geq 1$, $p>2$,  and $y\in \rrd$. There exists a solution $X^y=(X^y(t))_{t\geq 0}$ to the McKean--Vlasov SDE  \eqref{e1.15}--\eqref{e1.16} with $X^y(0) = y$ such that $\mathcal{L}_{X^y(t)} = w^y(t,x)dx$ for all $t>0$.
	\end{theorem}

\begin{kor}\label{c4.3} Let $s,t_0\ge0,\ y\in\rrd, \ \nu=w^y(t_0,x)dx$ $($with the convention $w^y(0,x)dx=\delta_y)$. There exists a solution $X^{s,(t_0,y)}=(X^{s,(t_0,y)}(t))_{t\ge s}$ to \eqref{e1.15}, \eqref{e1.16} on $[s,\9)$ such that $\mathcal{L}_{X^{s,(t_0,y)}(t)}=w^y(t_0+t-s,x)dx$ for all $t\ge s$.
\end{kor}

	\paragraph{Proof of Theorem \ref{t4.2}.} Since $w^y(t,x)=w^0(t,x-y)$, we only consider the case $y=0$.  
	We show that Theorem \ref{t3.3} applies to $w^0$. By Proposition \ref{p4.1}, $w^0$ is a weakly continuous probability solution to \eqref{e1.2}. Thus it is sufficient to prove $|\nabla w^0|^{p-2} \in L^1_{\textup{loc}}((0,\infty);W^{1,1}_{\textup{loc}}(\R^d))$ and \begin{equation}
	\label{e4.1'}
	|\nabla w^0|^{p-2} w^0,\ \nabla(|\nabla w^0|^{p-2})w^0 \in L^1_{\textup{loc}}([0,\infty);L^1(\R^d)).
	\end{equation}  Below let $C>0$  denote a constant possibly changing from line to line, only depending on $d, p, C_1$ and $T$.
	A direct calculation yields
	\begin{equation}\label{auxa}
	\nabla	w^0(t,x) = -Ct^{-k(1   +  \frac{p}{d(p-1)} )}\big(C_1- q t^{\frac{-kp}{d(p-1)}}|x|^{\frac{p}{p-1}}\big)_+^{\frac{1}{p-2}} |x|^{\frac{2-p}{p-1}}x,
	\end{equation}
	and, therefore, 
	\begin{equation}\label{e4.4}
	|\nabla w^0(t,x)|^{p-2} = Ct^{-k(p-2)(1+\frac{p}{d(p-1)})}\big(C_1-qt^{\frac{-kp}{d(p-1)}}|x|^{\frac{p}{p-1}}\big)_+ |x|^{\frac{p-2}{p-1}}.
	\end{equation}
	Due to Proposition \ref{p4.1} (ii) and since $\big(C_1-qt^{\frac{-k}{d}}|x|^{\frac{p}{p-1}}\big)_+$ is uniformly bounded on $(0,\infty)\times \R^d$, we find for each $T>0$
	\begin{align*}
	\int_0^T \int_{\R^d} |\nabla w^0|&^{p-2}(t,x)w^0(t,x) \,dx dt \leq C \int_0^T t^{-k(p-2)(1+\frac{p}{d(p-1)})} \int_{\{|x|\leq \beta t^{\frac k d}\}} |x|^{\frac{p-2}{p-1}}w^0(t,x)dx  \,dt 
	\\&\leq C \int_0^T \bigg( t^{-k(p-2)(1+\frac{p}{d(p-1)})  +   \frac k d \frac{p-2}{p-1}} \int_{\R^d} w^0(t,x)dx \bigg) dt = \int_0^T  t^{-k(p-2)(1+\frac{p}{d(p-1)})  +   \frac k d \frac{p-2}{p-1}} dt,
	\end{align*}
	where we set $\beta := (C_1q^{-1})^{\frac{p-1}{p}}$ and used that $w^0(t,x)dx$ is a probability measure for each $t >0$ for the last equality. Since $-k(p-2)(1+\frac{p}{d(p-1)})  +   \frac k d \frac{p-2}{p-1}  >-1$ for $p>2$ , $|\nabla w^0|^{p-2} w^0 \in L^1_{\textup{loc}}([0,\infty);L^1(\R^d))$ follows. Likewise, since even $-k(p-2)(1+\frac{p}{d(p-1)}) > -1$, we have $|\nabla w^0|^{p-2} \in L^1_{\textup{loc}}([0,\infty);L^1(\R^d))$.
	Furthermore,  \eqref{e4.4} implies $|\nabla w^0(t)|^{p-2}\in W^{1,1}_{\textup{loc}}(\R^d)$ for all $t>0$, with	
	\begin{equation*}
\barr{l}
	\nabla|\nabla w^0(t,x)|^{p-2} 
	\\	\quad 
	\dd= C t^{-k(p-2)(1+\frac{p}{d(p-1)})}\bigg[\frac{p-2}{p-1}\big(C_1-qt^{\frac{-kp}{d(p-1)}}|x|^{\frac{p}{p-1}}\big)_+|x|^{-\frac{p}{p-1}} x     -   \frac{qp}{p-1} t^{\frac{-kp}{d(p-1)}}x\mathds{1}_{|x| \leq \beta t^{\frac k d}}(x)\bigg].\earr
	\end{equation*}
	The estimate
	\begin{equation}\label{e4.5}
	|	\nabla|\nabla w^0(t,x)|^{p-2} |  \leq C \mathds{1}_{|x| \leq \beta t^{\frac k d}}(x) t^{-k(p-2)(1+\frac{p}{d(p-1)})}\big[ |x|^{-\frac{1}{p-1}} +  t^{\frac{-kp}{d(p-1)}}|x|   \big]
	\end{equation}
	implies $|\nabla w^0 |^{p-2} \in L^1_{\textup{loc}}([0,\infty);W^{1,1}_{\textup{loc}}(\R^d))$. 	
	Finally, \eqref{e4.5} also yields $$\nabla(|\nabla w^0|^{p-2})w^0 \in L^1_{\textup{loc}}([0,\infty);L^1(\R^d)).$$ Thus, Theorem \ref{t3.3} applies to $w^0$ and the assertion follows.\hfill$\Box$\bigskip

	For the convenience of the reader, in Appendix A we give more detailed calculations for some claims from the previous proof.

	\section{Nonlinear Markov processes and $p$-Brownian motion}\label{s5}
		
	\subsection{The classical case: Heat equation and Brownian motion}\label{s5.1}
	
	For $p=2$, \eqref{e1.2} is the heat equation \eqref{e1.1}, with fundamental solution $u^y$, $y\in\rrd$,  as in \eqref{e1.17}. 
	The kernels $u^y(t,x)dx$, $t \geq 0, y \in \R^d$, 
	where we set $u^y(0,x)dx := \delta_y$, determine a unique Markov process $(P_y)_{y\in \R^d} \subseteq \mathcal{P}(C([0,\infty),\R^d))$ with one-dimensional time marginals $(\pi_t)_*P_y = u^y(t,x)dx$ (for the definition of a Markov process, see Subsection \ref{s5.2} below). This Markov process is Brownian motion, $P_0$ is the {\it classical Wiener measure}, and $P_y$ is obtained from $P_0$ as
	$$P_y=(T_y)_* P_0,$$where $T_y:C([0,\9),\rrd)\to C([0,\9),\rrd)$, $T_y(\omega):=\omega+y$, is the translation of a path by $y$. 	
	Equivalently, any stochastic process $X^y=(X^y(t))_{t\ge0}$ on a probability space $(\Omega,\mathcal{F},\mathbb{P})$ with path law $(X^y)_*\mathbb{P}=P_y$ is   called Brownian motion with start in $y$ $(y=0$:  {\it standard Brownian motion}).	The associated SDE is
	\begin{equation*}
	dX^y(t) = dW(t), \quad t \geq 0, \quad X^y(0) = y,
	\end{equation*}
	which is probabilistically weakly well-posed, i.e. the solution path law of any solution equals $P_y$.
	
	\subsection{Linear and nonlinear Markov processes}\label{s5.2}
	The following definition of nonlinear Markov processes is inspired by McKean \cite{35a} and elaborated on in \cite{37a}. As mentioned in the introduction, it is the natural extension of the usual Markov property (satisfied by the family of solution path laws to a weakly well-posed (non-distribution dependent) SDE) to nonlinear FPEs and their associated McKean--Vlasov SDEs. For more details on the motivation and properties of nonlinear Markov processes, we refer to \cite{35a} and \cite{37a}. 
	
	Since equations \eqref{e1.2} and \eqref{e1.13} are time-dependent due to  the dependence of the coefficients on $\nabla u(t,x)$, we present a definition of nonlinear Markov processes in the time-inhomogeneous case. In this regard, we also refer to Remark \ref{r3.2} (i).

	\begin{dfn}\label{d5.1}
		Let $\mathcal{P}_0 \subseteq \mathcal{P}\ (=\mathcal{P}(\rrd))$. A \textit{nonlinear Markov process} is a family $(P_{s,\zeta})_{(s,\zeta)\in[0,\9)\times \mathcal{P}_0}$, $P_{s,\zeta} \in \mathcal{P}(C([s,\infty),\R^d)$, such that
		\begin{enumerate}
			\item[(i)] $(\pi^s_t)_*P_{s,\zeta}  =: \mu^{s,\zeta}_t \in \mathcal{P}_0$ for all $0\leq s \leq t$ and $\zeta \in \mathcal{P}_0$.
			\item[(ii)] The \textit{nonlinear Markov property} holds: for all $0\leq s \leq r \leq t$, $\zeta \in \mathcal{P}_0$
			\begin{equation}\label{MP}
			P_{s,\zeta}(\pi^{s}_{t} \in A|\mathcal{F}_{s,r})(\cdot) = p_{(s,\zeta),(r,\pi^s_r(\cdot))}(\pi^r_t\in A) \quad P_{s,\zeta}-\text{a.s.} \text{ for all }A \in \mathcal{B}(\R^d),
			\end{equation}
			where $(p_{(s,\zeta),(r,z)})_{z\in \R^d}$ is a regular conditional probability kernel from $\R^d$ to $\mathcal{B}(C([r,\infty),\R^d)$ of ${P}_{r,\mu^{s,\zeta}_r}[\,\cdot\,| \pi^r_r{=}z],\, z \!\in\! \R^d$ (i.e. in particular $p_{(s,\zeta),(r,z)}\! \in\! \mathcal{P}(C([r,\infty),\R^d)$ and $p_{(s,\zeta),(r,z)}(\pi^r_r {=} z) = 1$),\smallskip\ and $\mathcal{F}_{s,r}$ is the $\sigma$-algebra generated by $\pi^s_u$, $s\le u\le r$.
		\end{enumerate}
	\end{dfn}

The case of a classical (we also say {\it linear}) time-inhomogeneous Markov process is contained in the previous definition. In this case, 
$$\calp_0=\calp,\ P_{s,\zeta}=\int_\rrd P_{s,y}\,d\zeta(y),\ P_{s,y}:=P_{s,\delta_y},$$
and $p_{(s,y)(r,z)}=P_{r,z},$ for all $0\le s\le r,\ z\in\rrd,\ \zeta\in\calp$. In this case, \eqref{MP} is the standard time-inhomogeneous  Markov property. Since in this linear case the map $\zeta\mapsto P_{s,\zeta}$ is linear on its domain $\calp$ for every $s\ge0$, the measures $(P_{s,y})_{s\ge0,y\in\rrd}$ determine the entire Markov process and, moreover, the {\it Chapman--Kolmogorov equations} for the one-dimensional time marginals $\mu^{s,\zeta}_t$ (see (i) of the previous definition) hold, i.e.
\begin{equation}
\label{e5.2'}
\mu^{s,y}_t=\int_\rrd \mu^{r,z}_t d\mu^{s,y}_r(z),\ \ff 0\le s\le r\le t,\ y\in\rrd.
\end{equation}
In contrast to this, in the general nonlinear case the map $\zeta\to P_{s,\zeta}$ is not linear on $\calp_0$, even if $\calp_0=\calp$ (which we do not assume). As a consequence, one loses the Chapman--Kolmogorov equations, but the one-dimensional time marginals still satisfy the {\it flow property} 
\begin{equation}
\label{e5.1'}
\mu^{s,\zeta}_t=\mu^{r,\mu^{s,\zeta}_r}_t,\  \ff 0\le s\le r\le t,\ \zeta\in\calp_0,\end{equation} (in the linear case, this follows from \eqref{e5.2'}). 

The following theorem from \cite{37a} is the key for our main result of this section,  Theorem \ref{t5.2} below.

\begin{theorem}{\rm\cite[Thms. 3.4+3.8]{37a}}\label{t52}
	Let $\Pscr_0\subseteq \Pscr$ be a class of initial conditions and \mbox{$(\mu^{s,\zeta}_t)_{0\leq s \leq t, \zeta \in \Pscr_0} \!\subseteq\! \Pscr_0$} such that
	\begin{enumerate}
		\item[\rm(i)] For all $(s,\zeta) \in [0,\infty)\times \Pscr_0$, $(\mu^{s,\zeta}_t)_{t\geq s}$ is a weakly continuous probability solution to the nonlinear FPE \eqref{e1.3} with initial condition $(s,\zeta)$ in the sense of Definition {\rm\ref{dD.2}}, satisfying the flow property \eqref{e5.1'}.
		\item[\rm(ii)] For all $(s,\zeta) \in [0,\infty) \times \Pscr_0$, $(\mu^{s,\zeta}_t)_{t\geq s}$ is the unique weakly continuous probability solution to the linear FPE \eqref{eD.1} with coefficients $(t,x)\mapsto a_{ij}(t,x,\mu^{s,\zeta}_t)$ and $(t,x)\mapsto b_i(t,x,\mu^{s,\zeta}_t)$ and initial condition $(s,\zeta)$ in the sense of Definition {\rm\ref{dD.1}} in the class
	\begin{equation}\label{e5.3a}
	\{(\nu_t)_{t\geq s} \subseteq\mathcal{P} : \nu_t \leq C \mu^{s,\zeta}_t, \quad t\geq s,\text{ for some }C>0\}\end{equation}
	("restricted linearized distributional uniqueness").
	\end{enumerate}
	Then, for each $(s,\zeta) \in [0,\infty)\times \Pscr_0$ there is a unique solution path law $P_{s,\zeta}$ to the McKean--Vlasov SDE \eqref{e1.5}--\eqref{e1.6} such that
	\begin{equation}\label{tg}
	(\pi^s_t)_* P_{s,\zeta} = \mu^{s,\zeta}_t, \quad \forall 0\leq s \leq t, \zeta \in \Pscr_0,
	\end{equation}
	and the family $(P_{s,\zeta})_{s\geq 0, \zeta \in \Pscr_0}$ is a nonlinear Markov process in the sense of Definition {\rm\ref{d5.1}}.
	
	In particular, this nonlinear Markov process is uniquely determined by its one-dimensional time marginals $(\mu^{s,\zeta}_t)_{0\leq s \leq t, \zeta \in \Pscr_0}$ and equations  \eqref{e1.5}, \eqref{e1.6}.
\end{theorem}
More precisely, in the present work we will not apply the previous theorem, but the following corollary. The idea behind it is that the statement of the previous theorem remains true if the restricted linearized distributional uniqueness condition is known for fewer initial data, at the price of a uniqueness assertion for the corresponding solution path laws, which holds for fewer initial data as well.
\begin{kor} {\rm\cite[Corollary 3.10]{37a}} \label{c5.3}
	Let $\mathfrak{P}_0 \subseteq \Pscr_0 \subseteq \Pscr$ and suppose $(\mu^{s,\zeta}_t)_{0\leq s \leq t, \zeta \in \Pscr_0}$ satisfies {\rm(i)} from the previous theorem, but {\rm(ii)} only for $(s,\zeta) \in [0,\infty)\times \mathfrak{P}_0$. Further, assume $\mu^{s,\zeta}_t \in \mathfrak{P}_0$ for all triples $(s,t,\zeta)$ such that either $s \leq t$, $\zeta \in \mathfrak{P}_0$ or $s < t$, $\zeta \in \Pscr_0$. Then, there is a nonlinear Markov process $(P_{s,\zeta})_{s\geq 0, \zeta \in \Pscr_0} $ with \eqref{tg}, consisting of solution path laws to \eqref{e1.5}, \eqref{e1.6}. The uniqueness-assertion for $P_{s,\zeta}$ of Theorem {\rm\ref{t52}} remains true for initial data $(s,\zeta)$ with $\zeta\in\mathfrak{P}_0.$
\end{kor}

For the proof of Theorem \ref{t5.2}  below, we apply Corollary \ref{c5.3} to the nonlinear FPE \eqref{e1.13} and the corresponding McKean--Vlasov SDE \eqref{e1.15}, \eqref{e1.16}. Before we do this, let us describe the main idea of the proof of the general Theorem \ref{t52} from above.

\bigskip\noindent{\it Idea of proof of Theorem {\rm\ref{t52}}.} 
For the uniqueness assertion, let $(s,\zeta)\in [0,\infty)\times \Pscr_0$ and $P_1,P_2$ be two solution path laws to \eqref{e1.5}, \eqref{e1.6} with
\begin{equation}\label{H0}
	(\pi^s_t)_* P_i = \mu^{s,\zeta}_t, \quad \forall t \geq s, i \in \{1,2\}.
\end{equation}
It suffices to prove 
\begin{equation}\label{H1}
	\mathbb{E}_{P_1}[H] = \mathbb{E}_{P_2}[H]
\end{equation}
for all $H = \Pi_{i=1}^n h_i(\pi^s_{t_i})$, $n \in \N$, $s\leq t_1 < \dots < t_n$, $h_i: \R^d \to \R$ measurable, $c_i \leq h_i \leq C_i$ for some $0< c_i < C_i$ (indeed, the set of such $H$ is closed under pointwise multiplication and generates the Borel $\sigma$-algebra of $C([s,\infty),\R^d)$ with respect to the topology of locally uniform convergence). We call $n$ the \emph{length} of $H$. By \eqref{H0}, \eqref{H1} holds for all $H$ of length $1$, and we proceed via induction over $n$. Assuming \eqref{H1} is valid for all $H$ of length $n$, consider an arbitrary $H = \Pi_{i=1}^{n+1} h_i(\pi^s_{t_i})$ as above. By  \cite[Lemma 2.6]{43a}, the curves 
$$\eta^i_t := (\pi^s_{t_n})_*(\vrho P_i),\quad t \geq t_n,$$
where $\vrho$, defined via
$$\vrho := \frac{\Pi_{i=1}^nh_i(\pi^s_{t_i})}{\mathbb{E}_{P^1}[\Pi_{i=1}^nh_i(\pi^s_{t_i})] },$$
is a probability density with respect to $P_i$, $i \in \{1,2\}$, such that there is $0<c<C$ with \mbox{$c\leq \vrho \leq C$,} are weakly continuous probability solutions to the linear FPE with coefficients $a_{ij}(t,x,\mu^{s,\zeta}_t)$ and $b_i(t,x,\mu^{s,\zeta}_t)$ from the initial condition $(t_n, \eta)$ in the sense of Definition \ref{dD.1}, where 
$$\eta := (\pi^s_{t_n})_* (\vrho P_1) = (\pi^s_{t_n})_* (\vrho P_2) \in \Pscr.$$
Indeed, the latter equality follows from  the induction assumption. From the definition of $\vrho$ it follows that $(\eta^i_t)_{t\geq t_n}$, $i\in \{1,2\}$, belong to the set \eqref{e5.3a}, with $s$ and $\zeta$ replaced by $t_n$ and $ \mu^{s,\zeta}_{t_n}$, respectively. Moreover, the definition of $\vrho$ entails that $\eta$ is equivalent to $\mu^{s,\zeta}_{t_n}$ with a density bounded between two strictly positive constants. Due to the flow property, we have $\mu^{t_n,\mu^{s,\zeta}_{t_n}}_t = \mu^{s,\zeta}_t$, $t\ge t_n$, and so assumption (ii) and \cite[Lemma 3.7(ii)]{37a} imply $\eta^1_t = \eta^2_t$,  $t\ge t_n$, i.e. in particular $\eta^1_{t_{n+1}} = \eta^2_{t_{n+1}}$. Thus, we infer
\begin{align*}
	\frac{\mathbb{E}_{P^i}\big[\Pi_{i=1}^{n+1}h_i(\pi^s_{t_i})\big]}{\mathbb{E}_{P^1}\big[\Pi_{i=1}^{n}h_i(\pi^s_{t_i})\big]}
	= \int_{C([s,\infty),\R^d)} \vrho(w)&h_{n+1}(\pi^s_{t_{n+1}}(w)) \,dP^i(w)=\int_{\R^d} h_{n+1}(x)\,d\eta^i_{t_{n+1}}(x) ,
\end{align*}
which together with the induction assumption concludes the proof of the uniqueness assertion.

Regarding the nonlinear Markov property, for $0\leq s \leq r \leq t$ and $\zeta \in \Pscr_0$, it suffices to prove for all $A\in \Bscr(\R^d)$, $n \in \N$, $s\leq t_1<\dots < t_n\leq r$, and $h: (\R^d)^n \to \R$ measurable with $c\leq h \leq C$ for some $0<c<C$
\begin{align}\label{tgz}
\mathbb{E}_{s,\zeta}\big[h(\pi^s_{t_1},\dots,\pi^s_{t_n})\mathds{1}_{\pi^s_t \in A}\big] = \int_{C([s,\9),\rrd)}
p_{(s,\zeta),(r,\pi^s_r(\omega))}(\pi^r_t\in A)h(\pi^s_{t_1}(\omega),\dots,\pi^s_{t_n}(\omega)) \,dP_{s,\zeta}(\omega)
\end{align}
(where $\mathbb{E}_{s,\zeta}$ denotes expectation with respect to $P_{s,\zeta}$ and $p_{(s,\zeta),(r,z)}$, $z\in \R^d$, is as in Definition \ref{d5.1}, and to apply a monotone class argument. To prove \eqref{tgz}, define the probability measures $Q_g$ and $Q^\theta$ on $C([r,\infty),\R^d)$ as follows. First, define
\begin{equation*}
Q_g := \int_{\R^d} p_{(s,\zeta),(r,z)}g(z)\,d\mu^{s,\zeta}_r(z),
\end{equation*}
i.e.
$$Q_g(C) = \int_{\R^d} p_{(s,\zeta),(r,z)}(C) g(z)\, d\mu^{s,\zeta}_r(z),\quad \forall C \in \Bscr(C([r,\infty),\R^d)),$$
where $g = c_0\tilde{g}$, with $\tilde{g}:\R^d \to \R$ the $\mu^{s,\zeta}_r$-a.s. unique map such that
$$\mathbb{E}_{s,\zeta}\big[h(\pi^s_{t_1},\dots,\pi^s_{t_n})|\pi^s_r\big]= \tilde{g}(\pi^s_r)\quad \mathbb{P}_{s,\zeta}-a.s.$$ obtained from the factorization lemma 
and $c_0$ a normalizing constant such that $\int_{\R^d} g\,d\mu^{s,\zeta}_r = 1$.
Second, set
$$Q^\theta := (\Lambda^s_r)_*(\theta P_{s,\zeta}),$$
where $\Lambda^s_r := C([s,\infty),\R^d) \to C([r,\infty),\R^d)$, $\Lambda^s_r(\omega) := \omega_{|[r,\infty)}$ is the projection of a path $\omega = (\omega(t))_{t\geq s}$ to $(\omega(t))_{t\geq r}$,
and
$$\theta: C([s,\infty),\R^d) \to \R, \quad \theta := c_0 h(\pi^s_{t_1},\dots, \pi^s_{t_n})$$
with $c_0$ as above. By  \cite[Lemma 2.8]{43a} and \cite[Proposition 2.6]{43a}, $t\mapsto (\pi^r_t)_*Q_g$ and $t\mapsto (\pi^r_t)_*Q^\theta$, respectively, are weakly continuous probability solutions to the linear FPE with coefficients $(t,x)\mapsto a_{ij}(t,x,\mu^{s,\zeta}_t)$ and $(t,x)\mapsto b_i(t,x,\mu^{s,\zeta}_t)$ in the sense of Definition \ref{dD.1}, and their common initial condition is $(r,g \mu^{s,\zeta}_r)$. Since $g$ and $\theta$ are bounded between two strictly positive constants, both of the above $\Pscr$-valued curves belong to the set \eqref{e5.3a}, with $(s,\zeta)$ replaced by $(r,\mu^{s,\zeta}_r)$. Using the uniqueness assumption from (ii) for the initial condition $(r,\mu^{s,\zeta}_r)$ instead of $(s,\zeta)$ as well as \cite[Lemma 3.7 (ii)]{37a}, we obtain
$$(\pi^r_t)_*Q_g = (\pi^r_t)_*Q^\theta, \quad \forall t \geq r.$$
We note that the flow property of $(\mu^{s,\zeta}_t)_{s\leq t, \zeta \in \Pscr_0}$ was implicitly used several times in the above arguments.
Now we conclude \eqref{tgz} via the definitions of $Q^\theta$ and $Q_g$ by noting
\begin{align*}
\mathbb{E}_{s,\zeta}\big[h(\pi^s_{t_1},\dots,\pi^s_{t_n})\mathds{1}_{\pi^s_t \in A}\big] &= c_0^{-1}Q^\theta\circ (\pi^r_t)^{-1}(A) = c_0^{-1}Q_g\circ (\pi^r_t)^{-1}(A) \\&= c_0^{-1} \int _{{C([s,\infty),\R^d)}}p_{(s,\zeta),(r,\pi^s_r(\omega))}(\pi^r_t \in A)g(\pi^s_r(\omega))\,dP_{s,\zeta}(\omega) \\&= \int_{C([s,\infty),\R^d)}p_{(s,\zeta),(r,\pi^s_r(\omega))}(\pi^r_t\in A)h(\pi^s_{t_1}(\omega),\dots,\pi^s_{t_n}(\omega)) \,dP_{s,\zeta}(\omega).
\end{align*}

	
\subsection{$p$-Brownian motion}\label{s5.3}
Now, in analogy to the case $p=2$ presented in Section \ref{s5.1}, we implement Key Step 3 as follows: We consider the solutions $X^y=(X^y(t))_{t\ge0},$ $y\in\rrd$, to \eqref{e1.15}, \eqref{e1.16} constructed in Theorem \ref{t4.2}, more precisely their path laws, and prove that these laws are uniquely determined by $w^y(t,x),$ $y\in\rrd$, $t>0$, and \eqref{e1.15}, \eqref{e1.16}, as well as that they satisfy the nonlinear Markov property. To do so, we apply Corollary \ref{c5.3} to   \eqref{e1.13}, \eqref{e1.15}, \eqref{e1.16} and $w^y$, $y\in\rrd$, as follows.    Set	
$$\mathcal{P}_0 := \{w^y(\delta,x)dx,\  y\in \R^d,\ \delta \geq 0 \} \subseteq \mathcal{P},$$	with the convention $w^y(0,x)dx = \delta_y$. We note that  for each $\zeta \in \mathcal{P}_0$, the pair  $(\delta,y)\in [0,\infty)\times \R^d$ for the representation   $\zeta = w^y(\delta,x)dx$ is unique. 
	
	\begin{theorem}\label{t5.2} Let $d\ge2,\ p>2$.
		\begin{enumerate}
			\item [\rm(i)] Let $(s,\zeta) \in [0,\infty)\times \mathcal{P}_0$, $\zeta = w^y(\delta,x)dx$. The set of solution path laws to the McKean--Vlasov SDE  \eqref{e1.15}--\eqref{e1.16} with one-dimensional time marginals $w^y(\delta+t-s,x)dx$, $t \geq s$, and initial condition $(s,\zeta)$ contains exactly one element $P_{s,\zeta}$. The family $(P_{s,\zeta})_{s\geq 0, \zeta \in \mathcal{P}_0}$ is a nonlinear Markov process in the sense of Definition {\rm \ref{d5.1}}.   In particular, this nonlinear Markov process is uniquely determined by \eqref{e1.15}--\eqref{e1.16} and $\omega^y(t)$, $y\in\rrd$, $t\ge0$.
			\item[\rm(ii)] $(P_{s,\zeta})_{s \geq 0, \zeta \in \mathcal{P}_0}$ is time-homogeneous, i.e. $P_{s,\zeta} =(\hat{\Pi}_s^{-1})_*		
			 P_{0,\zeta}$ for all $(s,\zeta)\in[0,\9)\times \mathcal{P}_0$, where 
		\begin{equation}\label{e5.8'}
		\hat{\Pi}_s: C([0,\infty),\R^d) \to C([s,\infty),\R^d),\quad \hat{\Pi}_s: (\omega(t))_{t\geq 0} \mapsto (\omega(t-s))_{t\geq s}.\end{equation} Moreover, for $\zeta = w^y(\delta,x)dx$, we have $P_{0,\zeta} = ({\Pi}_{\delta}^{0})_*P_{0,y}$ $($the map $ \Pi^0_{\delta}$ is defined in \eqref{e5.3} below$)$.
		\end{enumerate}
	\end{theorem}
	
\begin{rem}\label{r5.5} For $\zeta=w^y(\delta,x)dx$, $P_{s,\zeta}$ from the previous theorem is the path law of the solution $X^{s,(\delta,y)}$ from Corollary \ref{c4.3}, i.e. when the latter is defined on the probability space $(\Omega,\mathcal{F},\mathbb{P})$, then
		$$P_{s,\zeta}=(X^{s,(\delta,y)})_*\mathbb{P}.$$
		\end{rem}

		To avoid confusion, we remark that in the following proof the times $s$ and $\delta$ are not related. In particular, for the initial condition  $\zeta = w^y(\delta,x)dx$, we do not only consider the initial pair $(\delta,\zeta)$, but necessarily \emph{any} $(s,\zeta)$, $s \geq 0$. 	Set 
		$$\mathfrak{P}_0 := \{w^y(\delta,x)dx, y \in \R^d, \delta>0\}=\calp_0\setminus\{\delta_y,\ y\in\rrd\}.$$ 	
		In the following proof, a crucial ingredient is Theorem \ref{t5.10}, which is formulated and proven in Section~\ref{s6a}.
	
	\medskip \noindent{\it Proof of Theorem}   \ref{t5.2}.\
		\begin{enumerate}
			\item [(i)] Setting, for $\zeta = w^y(\delta,x)dx$, $\delta\geq 0$, $y\in\rrd$,
			$$\mu^{s,\zeta}_t := w^y(\delta+t-s,x)dx, \quad t \geq s,$$
			it is straightforward to check that the family of probability measures $\{\mu^{s,\zeta}_t\}_{s\geq 0, t\geq s, \zeta \in \mathcal{P}_0}$ has the {flow property} \eqref{e5.1'}.    
	By Lemma \ref{l2.5}, Proposition \ref{p4.1} and \eqref{e4.1'},  $(\mu^{s,\zeta}_t)_{t\ge s}$ is a weakly continuous probability solution to the nonlinear FPE \eqref{e1.13} for each $(s,\zeta)\in[0,\9)\times\mathcal{P}_0$ in the sense of Definition \ref{dD.2}. Moreover, by Theorem \ref{t5.10} below, condition (ii) of Theorem \ref{t52} holds for all $(\mu^{s,\zeta}_t)_{t\ge s}$ with $(s,\zeta)\in[0,\9)\times\mathfrak{P_0}$. Since also the final condition of Corollary \ref{c5.3} is satisfied, we may apply Corollary \ref{c5.3} to obtain:				
			
	There is a nonlinear Markov process $(P_{s,\zeta})_{(s,\zeta)\in [0,\9)\times \mathcal{P}_0}$, $P_{s,\zeta} \in \mathcal{P}(C([s,\infty),\R^d)$, such that
\begin{enumerate}
\item [(I)] $(\pi^s_t)_*P_{s,\zeta} = \mu^{s,\zeta}_t, t\geq s$;
\item [(II)] $P_{s,\zeta}$ is a solution path law to the McKean--Vlasov SDE \eqref{e1.15}--\eqref{e1.16} on $[s,\9)$;
\item [(III)] For $s\geq 0$ and $\zeta \in \mathfrak{P}_0$, $P_{s,\zeta}$ is the \emph{unique} element from $\mathcal{P}(C([s,\infty),\R^d)$ with properties (I)--(II).
\end{enumerate}
Therefore, since $\mathcal{P}_0 \backslash \mathfrak{P}_0 = \{\delta_y, y \in \R^d\}$, it remains to prove the following claim.
			
\medskip\noindent{\bf Claim.} For $(s,y)\in [0,\infty)\times \R^d$, there is a \emph{unique} path law $P_{s,\delta_y}$ with properties (I)--(II).
			
\textit{Proof of Claim.} Let $P^1,P^2$ have properties (I)--(II) for $s\geq 0, \zeta = \delta_y$, $y \in \R^d$. We define for $ \omega = \omega(t)_{t\geq s}$ the map $\Pi^s_r : C([s,\infty),\R^d) \to C([s,\infty),\R^d)$ via
\begin{equation}\label{e5.3}
\Pi^s_r: \omega(t)_{t\geq s} \mapsto \omega(t+r)_{t\geq s}.
\end{equation}			
For any $r >0$, $i \in \{1,2\}$, we have
\begin{equation}\label{e5.9'}
(\pi^s_t)_*((\Pi^s_r)_*P^i)=(\pi^s_{t+r})_*P^i = \mu^{s,\delta_y}_{t+r} = w^y(t+r-s,x)dx,\quad \forall t \geq s,\end{equation}
where the second equality is due to (I). 
It is straightforward to check that $(\Pi^s_r)_*P^i  \!\in\! \mathcal{P}(C([s,\infty),\R^d)$, \mbox{$i \!\in\! \{1,2\}$,} is a solution path law to   the McKean--Vlasov SDE \eqref{e1.15}--\eqref{e1.16} with initial condition $(s,w^y(r,x)dx)$.
As $w^y(r,x)dx \in \mathfrak{P}_0$, \eqref{e5.9'} and (III) yield	$$(\Pi^s_r)_*P^1=(\Pi^s_r)_*P^2.$$
 Now let $s \leq u_1 < \dots < u_n$, $n \in \N$. First assume $u_1 >s$. Then for $i \in \{1,2\}$
\begin{align*}
(\pi^s_{u_1},...,\pi^s_{u_n})_*P^i
=(\pi^s_s,...,\pi^s_{u_n+s-u_1})_*((\Pi^s_{u_1-s})_*P^i),
\end{align*}
and by the previous part of the proof the right hand side  coincides for $i=1$ and $i=2$, since $u_1-s>0$. Now assume $s=u_1 <\dots < u_n$. Then, since $(\pi^s_s)_*P^i = \delta_y$, we find
\begin{align*}
		(\pi^s_{u_1},...,\pi^s_{u_n})_*P^i
		=\delta_y\otimes
		((\pi^s_{u_2},...,\pi^s_{u_n})_*P^i)	
			\end{align*}
			(where $\mu \otimes \nu$ denotes the product measure of the measures $\mu$ and $\nu$).
			Since $u_2>s$, the argument of the first case again yields that the right hand side  coincides for $i=1$ and $i=2$. Hence we have proven
			$$(\pi^s_{u_1},...,\pi^s_{u_n})_*P^1
			=(\pi^s_{u_1},...,\pi^s_{u_n})_*P^2$$
			for all $s\leq u_1 < \dots < u_n$, $n \in \N$, i.e. $P^1 = P^2$, which proves the claim and, thereby, the assertion.
			
			\item[(ii)] First note 
			$$(\pi^s_t)_*P_{s,\zeta} = \mu^{s,\zeta}_t = \mu^{0,\zeta}_{t-s} = (\pi^s_t)_*[(\hat\Pi_s)_*P_{0,\zeta}],\   \forall t \geq s,$$ with $\hat\Pi_s$ as in \eqref{e5.8'}.  
			Both $P_{s,\zeta}$ and $(\hat\Pi_s)_*P_{0,\zeta}$ are solution path laws to the McKean--Vlasov equation \eqref{e1.15}--\eqref{e1.16} on $[s,\9)$ with initial condition $(s,\zeta)$, hence, by (i),
			$$P_{s,\zeta} = (\hat\Pi_s)_*P_{0,\zeta}.$$ For the final statement, note that for $\zeta = w^y(\delta,x)dx$ the measures $P_{0,\zeta}$ and $(\Pi^0_{\delta})_*P_{0,y}$  have identical one-dimensional time marginals $w^y(\delta+t,x)dx,$ $ t \geq 0$, and both are solution path laws to \eqref{e1.15}--\eqref{e1.16} with initial condition $(0,w^y(\delta,x)dx)$. Thus the assertion follows from the uniqueness assertion in (i).\hfill$\Box$
		\end{enumerate}
 	
	\begin{rem}\label{r5.4}
		Theorem \ref{t5.2} (ii) implies that the nonlinear Markov process $(P_{s,\zeta})_{(s,\zeta) \in \R_+\times \mathcal{P}_0}$ is uniquely determined by $$(P_y)_{y \in \R^d}, \ P_y := P_{0,\delta_y}.$$ 
		Therefore, we also refer to $(P_y)_{y\in \R^d}$ as the \emph{unique  nonlinear Markov process	 determined by \eqref{e1.15}--\eqref{e1.16} and the one-dimensional time marginals $w^y(t_0), y \in \R^d, t_0 \geq 0$}.
	\end{rem}

In analogy to the linear case (discussed in Section \ref{s5.1}), we now define $p$-Brownian motion as the unique  nonlinear Markov process, as mentioned in the previous remark. 

\begin{dfn}\label{d5.5} Let $d\ge2$, $p>2$. We call the family of path laws $(P_y)_{y\in\rrd}$ from Remark \ref{r5.4} $p$-{\it Brownian motion}.
	\end{dfn}
		
\begin{remark}\label{r5.6a} \		
		\begin{enumerate}
\item [(i)]  Moreover, in analogy to the linear case $p=2$, we call $P_0$ the $p$-{\it Wiener measure}.
\item[(ii)] As in the case $p=2$, we also call any stochastic process $X^y=(X^y(t))_{t\ge0}$ on a probability space $(\Omega,\mathcal{F},\mathbb{P})$ a $p$-{\it Brownian motion with start in $y$},   if $(X^y)_*\mathbb{P}=P_y$, with $P_y$ as in Remark \ref{r5.4}. The one-dimensional time marginals of $X^y$ are $\mathcal{L}_{X^y(t)}=w^y(t,x)dx$, $t>0$, $\mathcal{L}_{X^y(0)}=\delta_y$.
\end{enumerate}
\end{remark}

\begin{rem}\label{r5.6}
		We point out that for $p>2$, unlike in the case $p=2$, the measures $P_y$ are not given as the image measure of $P_0$ under the translation map $T_y: C([0,\infty),\R^d) \to C([0,\infty),\R^d), T_y(\omega) := \omega + y$ (compare Section \ref{s5.1}). 
		This follows, for instance, from the fact that $(T_y)_*P_0$  is \emph{not} the solution path law of the McKean--Vlasov SDE \eqref{e1.15}--\eqref{e1.16} wtih initial condition $\delta_y$.
	\end{rem}

	\section{A crucial uniqueness result for a linearized $p$-Laplace equation}\label{s6a}
	
	As mentioned before, a crucial result used for the proof of Theorem \ref{t5.2} is a restricted distributional uniqueness result for a linearized $p$-Laplace equation. Here, we formulate and prove this result (see Theorem \ref{t5.10} below) and, thereby, complete   Key Step 3. 
	
	\subsection{A linearized uniqueness result}\label{s6.1}
	
	Here, we set for $\delta >0$,
	$$w_\delta(t,x) := w^0(t+\delta,x),\quad \vrho_\delta(t,x) := |\nabla w_\delta(t)|^{p-2}(x),\ (t,x)\in[0,\9)\times\rrd.$$ 
	The results of this subsection hold for $w^y$ instead of $w^0$ for every initial value $y \in \R^d$, but for simplicity we only mention the case $y=0$.
	On $Q_T := (0,T)\times \R^d$, where $T>0$ or $T=\infty$, we consider the linearized version of equation \eqref{e1.13} 
	\begin{equation}\label{e5.4}
	\frac d{dt}\, u = \Delta\big(\vrho_\delta u\big) - \divv\big(\nabla\vrho_\delta u\big),
	\end{equation}
	which is a linear Fokker--Planck equation obtained by fixing in the nonlinear equation \eqref{e1.13}  a priori the coefficients $\vrho_\delta$ and $\nabla \vrho_\delta$ in place of $|\nabla u|^{p-2}$ and $\nabla(|\nabla u|^{p-2})$, see also Appendix D.
	\begin{dfn}\label{d5.7}
		By a {\it distributional}  \textit{solution to \eqref{e5.4} with initial condition $\nu \in \mathcal{M}^+_b$} we mean a function $u \in L^1(Q_T)$ such that $t\mapsto u(t,x)dx$ is a weakly continuous curve of (signed) Borel   measures with
		\begin{equation*}
		\int_0^T \int_{\R^d} (\vrho_\delta + |\nabla \vrho_\delta|)|u|\,dx dt < \infty,
		\end{equation*}
		such that for all $\psi \in C^\9_0(\R^d)$
		\begin{equation}\label{e5.5}
		\int_{\R^d}\psi u(t)\,dx - \int_{\R^d} \psi \,d\nu = \int_0^t \int_{\R^d} (\vrho_\delta \Delta \psi + \nabla \vrho_\delta \cdot \nabla \psi ) \, u\, dx dt,\quad \forall0<  t < T.
		\end{equation}
		$u$ is called {\it probability solution} (instead of {\it distributional probability solution}), if $u\ge0$ and $\nu,u(t,x)dx\in\mathcal{P}$ for all $t\in(0,T)$.
	\end{dfn}
	\begin{rem}\label{r5.8}
		An equivalent condition to \eqref{e5.5} is
	\begin{equation}\label{e*}
		\int_{(0,T)\times \R^d} \bigg( \frac{d \vf}{dt} + \divv\big(\vrho_\delta \nabla \vf \big)\bigg)u\, dx dt + \int_{\R^d} \vf(0)\,d\nu = 0
	\end{equation}	
		for all $\vf \in C^{\9}_0([0,T)\times \R^d)$.
	\end{rem}

	Distributional solutions to \eqref{e5.4} are not necessarily related to $w_\delta$, but clearly $u(t,x)=w_\delta(t,x)$   is a distributional solution to \eqref{e5.4} in the sense of the previous definition with initial condition $w_\delta(0,x)dx = w^0(\delta,x)dx$.
	
	We collect some basic properties of $w_\delta$ and $\vrho_\delta$ used in the sequel.
	\begin{lem}\label{l5.9}
		Let $\delta >0$. Then $w_\delta$ and $\vrho_\delta$ are nonnegative functions with the following properties.
		\begin{enumerate}
			\item[\rm(i)] $w_\delta \in \bigcap_{R>0}C_0([0,R]\times \R^d)\cap L^\infty([0,\infty)\times \R^d)$.
			\item [\rm(ii)] $\vrho_\delta \in \bigcap_{R>0} C_0([0,R]\times \R^d)$ and $\{x\in \R^d: \vrho_\delta(t,x)>0\}=\{x\in \R^d: 0<|x| \leq \beta (t+\delta)^{\frac k d}\}$ for all $t\geq 0$, where $\beta = (\frac{C_1}{q})^{\frac{p-1}{p}}$.
			\item [\rm(iii)] $\nabla \vrho_\delta \in L^\infty_{\loc}([0,\infty), L^q(\R^d;\R^d))$ for $q \in [1,d(p-1))$.
		\end{enumerate}
	\end{lem}
	\begin{proof}
		The nonnegativity of $w_\delta$ and $\vrho_\delta$ as well as (i) and (ii) follow from their definitions. Regarding (iii), by taking into account \eqref{e4.5} and considering $d$-dimensional spherical coordinates, it suffices to note that $\int_0^1 r^{-\frac{q}{p-1}+d-1} dr <\infty$ holds. Indeed, $-\frac{q}{p-1}+d >0 \iff q < d(p-1).$ 
	\end{proof}

Comparing with Theorem \ref{t52} and the sets \eqref{e5.3a}, we are interested in uniqueness of distributional solutions to \eqref{e5.4} in the sense of Definition \ref{d5.7} in the class of probability solutions
	\begin{equation*}
	\Ascr_{\delta,T}:= 	\big\{u \in L^1\cap L^\infty(Q_T): t\mapsto u(t,x)dx \in C([0,T],\mathcal{P}), \, \exists C\geq 1: u(t,x)dx \leq C w_\delta(t,x)dx,\, \forall t \in [0,T]\big\},
	\end{equation*}
	where the inequality $u(t,x)dx\leq C w_\delta(t,x)dx$ is understood set-wise, i.e. 
	
	$$\int_A u(t,x)dx \leq C \int_A w_\delta(t,x)dx,\quad \forall A \in \Bscr(\R^d),$$
	and $C([0,T],\mathcal{P})$ is the set of weakly continuous maps $t\mapsto\mu_t\in\mathcal{P}$. 
	If $T=\9$, the interval $[0,T]$ in the definition of $\mathcal{A}_{\delta}$ is understood as $[0,\9)$. 
	\begin{remark}\label{r5.10} Clearly, $w_\delta \in \Ascr_{\delta,T}$.
	Note that $u \in \Ascr_{\delta,T}$ implies $u(t,x)dx = \big(\eta(t,x)w_\delta(t,x)\big)dx$ with $0\leq \eta \leq C$ for some $C\ge1$  $dxdt$-a.s. and, for each $t \in [0,T]$, $\supp \eta(t) \subseteq \supp w_\delta(t)$.
\end{remark}
The following result was crucially used in  the proof of Theorem \ref{t5.2}. 
	
	\begin{theorem}\label{t5.10}
		Let $d \geq 2$, $p > 2$, $\delta>0$,  $0<T<\infty$ and $v_1\in \Ascr_{\delta,T}$ be a distributional solution to equation \eqref{e5.4} in the sense of Definition {\rm\ref{d5.7}} with initial condition $w_\delta(0,x)dx = w^0(\delta,x)dx$. Then $v_1 = w_\delta$ $dxdt$-a.s. on $Q_T$. In particular, $w_\delta$ is the only distributional solution to \eqref{e5.4} on $(0,\infty)\times \R^d$ in $\Ascr_{\delta,\infty}$ with initial condition  $w_\delta(0,x)dx$.	\end{theorem}
As explained in Section \ref{s5}, for the proof of Theorem \ref{t5.2}, we need the result of Theorem \ref{t5.10} for all initial times $s\ge0$. More precisely, for $s\ge0$, one considers $w_\delta(t-s)$ and $\zeta_\delta(t-s)$ instead of $w_\delta(t)$ and $\vrho_\delta(t)$ in equation \eqref{e5.4} and considers \eqref{e5.4} on $Q_{s,T}:=(s,T)\times\rrd$ instead of $Q_T$. It is obvious how to extend Definition \ref{d5.7} and the definition of the sets $\mathcal{A}_{\delta,T}$ in this regards, and also that Theorem \ref{t5.10} holds accordingly.

\begin{remark}\label{r6.6} \
	\begin{itemize}
	  \item[(i)] We would like to mention here that Theorem \ref{t5.10} appears to be really new, though it is a uni\-que\-ness result for a linear Fokker--Planck equation which, however, is degenerate. For uni\-que\-ness results for the latter we refer to \cite[Section 9.8]{9a} and the references therein. It should be noted that \mbox{\cite[Theorem 9.8.2]{9a}} (see also Theorem 1 and its complete proof in the original work \cite{10a}) is close to what we need for our case. But it requires that $\vrho^{-\frac12}_\delta\nabla\vrho_\delta$ is bounded, which, however, is not true in our case, so it does not apply. Our proof of Theorem \ref{t5.10} below is strongly based on the explicit form of $w_\delta$ and $\vrho_\delta$. In this respect we, in particular, refer to the proof of Claim 4 below.
\item[(ii)]  As an easy consideration shows, the assertion of Theorem \ref{t5.10} is equivalent to the statement that $w_\delta$ is an extreme point in the convex set of all probability solutions to \eqref{e5.4} with initial condition $w^0(\delta,x)$, see also \cite[Lemma 3.5]{37a}.
\end{itemize}
\end{remark}

	\subsection{Proof of Theorem \ref{t5.10}}\label{s6.2}
	
	Let $T\in(0,\9),\ \delta>0$, $d\ge2$, and $p>2$. 

\begin{remark}\label{r6.2} We note that in Remark  \ref{r5.8}, since
		$${\rm div}(\vrho_\delta\nabla  \vf)=\vrho_\delta\Delta \vf+\nabla\vrho_\delta
		\cdot\nabla \vf,$$
		by a standard localization argument we can replace $C^\9_0([0,T)\times\rrd)$ by
		$$C^\9_{0,b}([0,T)\times\rrd):=\{\vf\in C^\9_b([0,T)\times\rrd)\mid\exists\lbb>0\mbox{ such that }\vf(t,x)=0,\ff(t,x)\in[T-\lbb,T)\times\rrd\}.$$For the convenience of the reader we include a proof in Appendix B.
	\end{remark}

	\begin{proof}[Proof of Theorem {\rm\ref{t5.10}}] Let $v_1\in\mathcal{A}_{\delta,T}$, i.e. there is $\eta\in L^1(Q_T)\cap L^\9(Q_T)$ such that $\eta\ge0
		$,\break $|\eta(t)w_\delta(t)|_{L^1(\rrd)}=1$ for all $t\in[0,T]$, and $v_1=\eta w_\delta$.   Define $$v:=w_\delta-v_1=(1-\eta)w_\delta.$$ 
		Note that $v$ is a (signed) distributional solution to \eqref{e5.4} in the sense of Definition \ref{d5.7} with initial condition $v(0)\equiv0$. Let $f\in C^\9_0(Q_T)$ and  for $\vp\in(0,1)$  consider the equation
\begin{equation}\label{e6.4z}
		\barr{l}
		\dd\frac{\pp\vf_\vp}{\pp t}+{\rm div}((\rho_\delta+\vp)\nabla\vf_\vp)=f\mbox{\ \ in }Q_T\vspace*{1,5mm}\\
		\vf_\vp(T,x)=0,\ \ x\in\rrd.
		\earr\end{equation}
		By standard existence theory for linear parabolic equations (see, e.g., \cite{10'}, p.~345) it follows that equation \eqref{e6.4z}) has a unique solution
		$$\vf_\vp\in C([0,T];L^2)\cap L^2(0,T;H^1),$$
		with $\frac{d\vf_\vp}{dt}\in L^2(0,T;H\1)$. Moreover, we have
		\begin{equation}\label{e6.5z}
			|\vf_\vp(t)|^2_2+\int^T_0\int_\rrd(\rho_\delta(t,x)+\vp)|\nabla\vf_\vp(t,x)|^2dtdx
		\le\dd\int^T_0\int_\rrd|f(t,x)|^2dtdx.
	\end{equation}
In particular, \eqref{e6.4z} is equivalent to
	\begin{equation}\label{e6.5}
\barr{l}
\dd\frac{\pp\vf_\vp}{\pp t} +(\vrho_\delta+\vp)\Delta\vf_\vp+\nabla
\vrho_\delta\cdot\nabla\vf_\vp=f\mbox{\ \ on }Q_T,\vspace*{1,5mm}\\
\vf_\vp(T,x)=0,\ x\in\rrd.\earr
\end{equation}	\begin{claim}\label{claim1a}  We have $\vf_\vp\in W^{2,2}(Q_T)$, that is,
	\begin{equation}\label{e6.6z}
		\frac{\pp\vf_\vp}{\pp t},\frac\pp{\pp x_i}\,\vf_\vp,\ 
		\frac\pp{\pp x_i}\ \frac\pp{\pp x_j}\,\vf_\vp\in L^2(Q_T),\ i,j=1,...,d.\end{equation}
		\end{claim}		
\medskip\noindent{\it Proof of Claim} \ref{claim1a}.
	We set $g_\vp=\nabla\rho_\delta\cdot\nabla
	\vf_\vp.$ By \eqref{e6.5z} we have $\vf_\vp,|\nabla\vf_\vp|\in L^2(Q_T)$. By Lemma \ref{l5.9} (iii), we know that $\nabla\rho_\delta\in L^\9(0,T;L^q)$, $q\in[1,d(p-1))$. Fix $q\in(d,d(p-1))$. Then, by H\"older's inequality,
	$$g_\vp\in L^{\gamma_1}(Q_T),\ \mbox{ for }\gamma_1:=\frac{2q}{1+q}\in(1,2).$$  
	Taking into account that, by \eqref{e6.5},
	$$\frac{\pp\vf_\vp}{\pp t}+(\rho_\vp+\vp)\Delta\vf_\vp=f-g_\vp\in L^{\gamma_1}(Q_T),$$
	we get by \cite[Theorem 9.1, p.~341]{31a} that $\vf_\vp\in W^{2,\gamma_1}(Q_T)$,   that is,
	$$\frac{\pp\vf_\vp}{\pp t},\frac{\pp\vf_\vp}{\pp x_1},\frac\pp{\pp x_i}\cdot\frac\pp{\pp x_j}\,\vf_\vp\in L^{\gamma_1}(Q_T),\ i,j=1,...,d.$$
	On the other hand, by the Sobolev--Nirenbergh--Gagliardo theorem (see, e.g., \cite{10'}, p.~283) we have
$$|\nabla\vf_\vp|\in L^{\alpha_1}(Q_T)\mbox{\ \ for }\alpha_1:=\frac{\gamma_1d}{d-\gamma_1}$$
and this yields as above
$$g_\vp\in L^{\gamma_2}(Q_T),\ \mbox{ for }  \gamma_2:=\frac{\alpha_1q}{\alpha_1+q}\in(1,\9).$$
Then, again by Theorem 9.1 in \cite{31a}, it follows that $\vf_\vp\in W^{2,\gamma_2}(Q_T)$ and, therefore,
 $$|\nabla\vf_\vp|\in L^{\alpha_2}(Q_T)\ \mbox{ for }\alpha_2:=\frac{\gamma_2d}{d-\gamma_2}.$$
 Continuing, we obtain sequences $(\gamma_1)_{i\in\mathbb{N}},$ $(\alpha_i)_{i\in\mathbb{N}}$ such that
 \begin{equation}
 \label{e6.8z}
 g_\vp\in L^{\gamma_i}(Q_T),\ \vf_\vp\in W^{2,\gamma_i}(Q_T),\ \ff i\in\mathbb{N},
 \end{equation}
 and
 $$\barr{l}
 \gamma_1=\dd\frac{2q}{2+q},\ \gamma_{i+1}=\dd\frac{\alpha_iq}{\alpha_i+q},\ i\in\mathbb{N},\vspace{1,5mm}\\ 
\alpha_i=\dd\frac{\gamma_id}{d-\gamma_i},\ i\in\mathbb{N},\mbox{ as long as }\gamma_i<d.\earr$$
This yields the recursive formula
$$ \gamma_1=\dd\frac{2q}{2+q}\in(1,2),\ \gamma_{i+1}=\gamma_i\,\dd\frac{dq}{\gamma_id+q(d-\gamma_i)},\  i\in\mathbb{N}, \mbox{ as long as }\gamma_i<d.$$
Since $q>d$, we have for all $i\in\mathbb{N}$ with $\gamma_i<d$ that 
\begin{equation}
\label{e6.9z}
\gamma_{i+1}=\gamma_i\,\frac{dq}{dq+\gamma_i(d-q)}>\gamma_i.
\end{equation}
Suppose that
\begin{equation}
\label{e6.10z}
\gamma_i\le2\mbox{ for all }i\in\mathbb{N}.\end{equation}
Then, by \eqref{e6.9z} there exists
$$\gamma:=\lim_{i\to\9}\gamma_i\le2$$
and, passing to the limit in \eqref{e6.9z}, we obtain
$$\gamma=\gamma\,\frac{dq}{dq+\gamma(d-q)}>\gamma.$$
This contradiction implies that \eqref{e6.10z} is wrong, so there exists $i\in\mathbb{N}$ such that $\gamma_i>2$ and Claim \ref{claim1a} follows by \cite[Theorem 9.1]{31a}, because $g_\vp\in L^2(Q_T)$ by \eqref{e6.8z} and interpolation, since $\gamma_1\in(1,2).$ \hfill$\Box$\bigskip
  
By the maximum principle
		\begin{equation}\label{e6.9}
		\sup_{\vp\in(0,1)}|\vf_\vp|_{L^\9(Q_T)}\le C_0|f|_{L^\9(Q_T)}.\end{equation}
		This follows in a standard way from \eqref{e6.5} by multiplying the equation with $(\vf_\vp-(T-t)|f|_{L^\9(Q_T)})_+$ and $(\vf_\vp+(T-t)|f|_{L^\9(Q_T)})_-$, respectively, and integrating over  $Q_T$.  
		
		Setting $\vrho^\vp_\delta:=\vrho_\delta+\vp$ and multiplying \eqref{e6.5} by $-\vf_\vp$, integrating over $(t,T)\times\rrd$ and using that $\vf_\vp(T,\cdot)=0$, we obtain by Gronwall's lemma (see \eqref{e6.5z} 
		\begin{equation}\label{e6.10}
		|\vf_\vp(t)|^2_{L^2(\rrd)}+
		\int^T_t\int_\rrd\vrho^\vp_\delta|\nabla\vf_\vp|^2dxds
		\le C_1\dd\int^T_t\int_\rrd|f|^2dxds,\ \ff t\in(0,T).\end{equation}		
		
		Now, we define for $\lbb\in(0,1)$
		$$\vf^\lbb_\vp(t,x)=\eta_\lbb(t)\vf_\vp(t,x),\ (t,x)\in Q_T,$$
		where $\eta_\lbb(t)=\eta\left(\frac t\lbb\right)\eta\left(\frac{T-t}\lbb\right)$ and 
		$\eta\in C^2([0,\9))$ is such that
		$$\eta(r)=0\mbox{ for }r\in[0,1],\ \eta(r)=1\mbox{ for }r>2.$$
	
		Below, we abbreviate $\frac d{dt}\,f$ by $f_t$.
		
		\begin{claim}\label{claim}    
			We have
			\begin{equation}\label{e6.11}
			\int^T_0{}_{H^1}\left<\vrho^\vp_\delta(t)
			\nabla\vf_\vp(t),\nabla(\vf_\vp)_t(t)\right>_{H^{-1}}dt
			\le-\frac12\int_{Q_T}(\vrho_\delta)_t|\nabla\vf_\vp(t,x)|^2dtdx,
			\end{equation}where ${}_{H^1}\left<\cdot,\cdot\right>_{H^{-1}}$ is the duality pairing on $H^1(\rrd;\rrd)\times H^{-1}(\rrd;\rrd)$.
		\end{claim}
		
		\medskip\noindent{\it Proof of Claim} \ref{claim}. 
		We set $p^*:=\frac{2d}{d-2},\ q^*:=d$ if $d\ge3$, and $p^*:=q^*:=4$ if $d=2$, and 	
		note that by \eqref{e6.5}--\eqref{e6.5z} we have    $\nabla(\vf_\vp)_t\in L^2((0,T);H\1)$ and,  for a.e.  $t\in(0,T)$,
		\begin{align*}
		|\vrho^\vp_\delta(t)\nabla\vf_\vp(t)|_{H^1} 
		&\le|\nabla\vrho_\delta(t)\cdot
		\nabla\vf_\vp(t)|_{L^2}
		+|\vrho^\vp_\delta(t)\Delta\vf_\vp(t)|_{L^2}\\
		&\le|\vrho^\vp_\delta(t)|_{L^\9}|\Delta\vf_\vp(t)|_{L^2}
		+\(\int_\rrd|\nabla\vrho_\delta(t,x)|^2|\nabla\vf_\vp(t,x)|^2dx\)^{\frac12}\\
		&\le|\vrho^\vp_\delta(t)|_{L^\9}|\Delta\vf_\vp(t)|_{L^2}
		+C_2|\nabla\vf_\vp(t)|
		_{L^{p^*}}|\nabla\vrho_\delta(t)|_{L^{q^*}}
		\le C_3|\vf_\vp(t)|_{H^2},
		\end{align*}
		because, by Lemma \ref{l5.9} (iii), $|\nabla\vrho_\delta|_{L^{q^*}}\in L^\9(0,T)$.  Since  by \eqref{e6.6z} yields  $|\vf_\vp|_{H^2}\in L^2(0,T)$, we infer that $\vrho^\vp_\delta\nabla\vf_\vp\in L^2((0,T);H^1)$ and so  the left hand side  of \eqref{e6.11}  is well defined. We choose a sequence $\{\vf^\vp_\nu\}\subset C^1([0,T];H^1)$ such that $\vf^\vp_\nu(T,\cdot)=0$ and, for $\nu\to0$,
		\begin{equation}\label{e6.12}
		\barr{rcll}
		\nabla\vf^\vp_\nu&\to&\nabla\vf_\vp&\mbox{ strongly in }L^2((0,T);H^1)\\
		\nabla(\vf^\vp_\nu)_t&\to&\nabla(\vf_\vp)_t&\mbox{ strongly in }L^2((0,T);H\1)
		\earr
		\end{equation}
		Such a sequence might be, for instance, 
		$$\vf^\vp_\nu(t)=(\vf_\vp*\theta_\nu)(t)-
		(\vf_\vp*\theta_\nu)(T),$$ where $\theta_\nu=\theta_\nu(t)$, $\nu>0$, is a standard mollifier sequence on $\rr$. Here, for technical purposes, we define $\vf_\vp(r)$ by $\vf_\vp(0)
		$ and $\vf_\vp(T)$ for $r\in(-1,0)$ and $r\in(T,T+1)$, respectively. 		
		Then, since $\nabla\vf^\vp_\nu(T)=0$, we have
		\begin{align*}
		\int^T_0{}_{H^{1}}\left<
		\vrho^\vp_\delta(t)\nabla\vf^\vp_\nu(t),
		\nabla(\vf^\vp_\nu(t))_t\right>_{H^{-1}}dt
		=\frac12\int_{Q_T}
		\vrho^\vp_\delta(t,x)
		|\nabla\vf^\vp_\nu(t,x)|^2_t dtdx
		\\
		=-\frac12\int_{\rrd}
		\vrho^\vp_\delta(0,x)
		|\nabla\vf^\vp_\nu(0,x)|^2 dx
		-\frac12\int_{Q_T}
		\vrho_\delta(t,x)
		|\nabla\vf^\vp_\nu(t,x)|^2 dtdx\\
		\le-\frac12\int_{Q_T}(\vrho_\delta)_t
		(t,x)
		|\nabla\vf^\vp_\nu(t,x)|^2dtdx.
		\end{align*}
	Letting $\nu\to0$,   we get by 
		\eqref{e6.12} that \eqref{e6.11} holds, as claimed.\hfill$\Box$\bigskip
		
		Now, by \eqref{e6.5} and \eqref{e6.11}, we have
		\begin{align*}
		\int_{Q_T}|(\vf_\vp)_t(t,x)|^2dtdx
		&=\int^T_0{}_{H^1}\!\!\left<	\vrho^\vp_\delta(t)\nabla\vf_\vp(t),
		\nabla(\vf_\vp)_t(t)\right>_{H^{-1}}dx
		+\int_{Q_T}f(t,x)(\vf_\vp)_t(t,x)dtdx\\
		&\le-\frac12\int_{Q_T}(\vrho_\delta)_t(t,x)
		|\nabla\vf_\vp(t,x)|^2dtdx
		+\int_{Q_T}f(t,x)(\vf_\vp)_t(t,x)dtdx.
		\end{align*}
		Since $(\vrho_\delta)_t\in L^\9(Q_T)$, this implies
		$$\int_{Q_T}|(\vf_\vp)_t|^2dtdx\le-\int_{Q_T}(\vrho_\delta)_t|\nabla\vf_\vp|^2dtdx+\int^T_0|f(t)|^2_{L^2(\rrd)}\,dt.$$
		By the definition of $\vrho_\delta$ we have 
\begin{equation}\label{e6.13}	-(\vrho_\delta)_t
\le C_2\delta^{-1}\vrho_\delta,\end{equation}		
		and hence we obtain	
		\begin{equation}\label{e6.14}
		\int_{Q_T}|(\vf_\vp)_t|^2dtdx\le\int_{Q_T}\vrho_\delta|\nabla\vf_\vp|^2dtdx+
		\int^T_0|f(t)|^2_{L^2(\rrd)}\,dt.\end{equation}	
		Then, by \eqref{e6.5} we have
		\begin{equation}\label{e6.15}
		\begin{array}{l}
		\dd\frac\pp{\pp t}\ \vf^\lbb_\vp+\divv((\vrho_\delta+\vp)\nabla\vf^\lbb_\vp)=f\eta_\lbb+\eta'_\lbb\vf_\vp\mbox{ on }Q_T \vspace*{1mm}\\
		\vf^\lbb_\vp(T,x)=0,\ \ff x\in\rrd.\end{array}
		\end{equation}
		By \eqref{e6.5}--\eqref{e6.9}  we have $\vf^\lbb_\vp\in W^{2,1}_2(Q_T)$ and  
		\begin{eqnarray}
		&\vf^\lbb_\vp\in L^2((0,T);H^2)\cap C([0,T];L^2),\ \dd\frac d{dt}\,\vf^\lbb_\vp\in L^2((0,T);L^2),\ \vf^\lbb_\vp\in L^\9(Q_T).\label{e6.16}	\end{eqnarray}
		By Remark \ref{r6.2}, $v$ satisfies \eqref{e*} for all $\vf\in C^\9_{0,b}([0,T)\times\rrd)$,    
		and $\nu=v(0,x)dx$ is the zero measure. Hence, we infer by density that $v$  satisfies \eqref{e*} also for all  functions $\vf=\vf^\lbb_\vp$ with properties \eqref{e6.16}. 
		Therefore,  we have
		\begin{equation}\label{e6.17} \int_{Q_T}v\left(\frac{\pp\vf^\lbb_\vp}{\pp t}+\divv(\vrho_\delta\nabla\vf^\lbb_\vp)\right)dtdx=0,\ \ff\vp,\lbb\in(0,1) 
		\end{equation}(see Appendix C for details). 
		
		Next, we get by \eqref{e6.15} the following equality
		$$\dd\frac12\,|\vf^\lbb_\vp(t)|^2_{L^2(\rrd)}+\vp\int^T_t|\nabla\vf^\lbb_\vp(s)|^2_{L^2(\rrd)}ds+\int^T_t\int_\rrd\vrho_\delta|\nabla\vf^\lbb_\vp|^2dsdx
		=-\dd\int^T_t\int_\rrd(f\eta_\lbb+\eta'_\lbb\vf_\vp)\vf^\lbb_\vp dtdx.$$
		Taking into account \eqref{e6.9}, we get
		\begin{equation}\label{e6.18}
		|\vf^\lbb_\vp(t)|^2_{L^2(\rrd)}+\vp\int^T_t
		|\nabla\vf^\lbb_\vp(s)|^2_{L^2(\rrd)}ds
		+\int_{Q_T}\vrho_\delta|\nabla\vf^\lbb_\vp|^2dtdx\le C_3,\ \ff\vp,\lbb\in(0,1).
		\end{equation}
		Now, we have as in \eqref{e6.11} that
		\begin{equation*}
		\int^T_0{}_{H^1}\left<\vrho^\vp_\delta(t)
		\nabla\vf^\lbb_\vp(t),
		\nabla(\vf^\lbb_\vp(t))_t\right>_{H^{-1}}\,dt
		=-\frac12\int_{Q_T}
		(\vrho_\delta)_t(t,x)
		|\nabla\vf^\lbb_\vp(t,x)|^2dtdx, 
		\end{equation*}
		and this yields
		\begin{equation*}
		\dd\int_{Q_T}|(\vf^\lbb_\vp)_t|^2dtdx
		\le-\int_{Q_T}(\vrho_\delta)_t
		|\nabla \vf^\lbb_\vp|^2 dtdx
		+C_4\dd\int_{Q_T}(|f\eta_\lbb|^2+
		|\eta'_\lbb\vf_\vp|^2) dtdx,\ \ff\vp\in(0,1). 
		\end{equation*}
		Then, by \eqref{e6.10}, \eqref{e6.13}  and \eqref{e6.18}  we get
		\begin{equation}\label{e6.19}
		\barr{r}
		\dd\int_{Q_T}|(\vf^\lbb_\vp)_t|^2dtdx
		\le\int_{Q_T}\vrho_\delta|\nabla\vf^\lbb_\vp|^2dtdx+
		C_4\int_{Q_T}
		(|f|^2+\lbb^{-1}|\vf_\vp|^2)dtdx\le \(1+\frac1\lbb\)C_5,\\ \ff\vp,\lbb\in(0,1).\earr\end{equation}
		(Here, we have denoted by $C_i$, $i=0,1,2,...$, several positive constants independent of $\vp,\lbb$.) Now, we fix a sequence $\lbb_n\to0$. 
		Hence, by \eqref{e6.13}, \eqref{e6.14}, \eqref{e6.18} and \eqref{e6.19} and a diagonal argument, we  find a subsequence $\vp_k\to0$ such that, for all $n\in\mathbb{N}$ as $k\to\9$, 
		\begin{equation}
		\label{e6.20}
		\left\{
		\begin{array}{rcll}
		\vf_{\vp_k}\ \to\ \vf,\ \vf^{\lbb_n}_{\vp_k}&\to&\vf^{\lbb_n}&\mbox{weakly in }L^2(Q_T)\medskip\\ 
		(\vf_{\vp_k})_t\ \to\ (\vf)_t,\ (\vf^{\lbb_n}_{\vp_k})_t&\to&(\vf^{\lbb_n})_t&\mbox{weakly in }L^2(Q_T)\medskip\\ 
		\divv((\vrho_\delta+\vp_k)\nabla\vf^{\lbb_n}_{\vp_k})&\to&\zeta_{\lbb_n}&\mbox{weakly in }L^2(Q_T).
		\end{array}\right.
		\end{equation}
		To finish the proof of Theorem \ref{t5.10}, we need the following three claims.\hfill$\Box$
		
		\begin{claim}\label{claim1} The maps $t\mapsto w_\delta(t)$ and $t\mapsto v_1(t)$ are weakly continuous from $[0,T]$ to $L^2(\rrd)$.\end{claim}
			
		\medskip\noindent{\it Proof of Claim {\rm\ref{claim1}}.} Since $\sup\limits_{t\in[0,T]}\{|w_\delta(t)|_{L^1(\rrd)}, |w_\delta(t)|_{L^\9(\rrd)}\}<\9$, also $\sup\limits_{t\in[0,T]}\{|w_\delta(t)|_{L^2(\rrd)}\}<\9.$ Hence, for every sequence $\{t_n\}_{n\in\mathbb{N}}\subseteq[0,T]$ with limit $t\in[0,T]$ there is a subsequence $\{t_{n_k}\}_{k\in\mathbb{N}}$ such that $w_\delta(t_{n_k})$ has a weak limit in $L^2(\rrd)$.   Due to the weak continuity (in the sense of measures) of $t\mapsto w_\delta(t,x)dx$, this limit is $w_\delta(t)$, which implies the $L^2(\rrd)$-weak continuity of $[0,T]\ni t\mapsto w_\delta(t).$ The same argument applies to $v_1$.\hfill$\Box$
		
			\begin{claim}\label{claim2} {\it We have}
			\begin{equation}\label{e6.21}
			\lim_{n\to\9}\ \lim_{k\to\9}
			\int_{Q_T} v\eta'_{\lbb_n}\vf_{\vp_k}\,dtdx=0.
			\end{equation}
		\end{claim}

		\noindent{\it Proof of Claim {\rm\ref{claim2}}.} We first note that by \eqref{e6.20} for every $n\in\mathbb{N}$
		$$\lim_{k\to\9}\int_{Q_T}  v\eta'_{\lbb_n}\vf_{\vp_k}\,dtdx
		=\int_{Q_T}  v\eta'_{\lbb_n}\vf\,dtdx,$$
		and that $\vf\in C([0,T];L^2)$. Furthermore,  by \eqref{e6.10}, $\vf(T,\cdot)=0$. Then, for every $\lbb\in(0,1)$,
		$$\begin{array}{lcl}
		\dd\int_{Q_T}  v\eta'_\lbb\vf\,dtdx
		&=&\dd\frac1\lbb\int^{2\lbb}_\lbb\eta'
		\left(\frac t\lbb\right)
		\int_{\rr^d}  v(t,x)\vf(t,x)dtdx
		+\dd\frac1\lbb\int^{T-\lbb}_{T-2\lbb}\eta'\left(\frac{T-t}\lbb\right)\int_{\rr^d}  v(t,x)\vf(t,x)dx\smallskip\\
		&=&\dd\dd\int^2_1\eta'(\tau)d\tau\int_{\rr^d}  v(\lbb\tau,x)\vf(\lbb\tau,x)dx
		+\dd\int^2_1\eta'(\tau)d\tau\int_{\rr^d}   v(T-\lbb\tau,x)\vf(T-\lbb\tau,x)dx,
		\end{array}$$where, as $\lbb\to0$,  both terms converge to zero by Claim \ref{claim1}  and Lebesgue's dominated convergence theorem.\hfill$\Box$ 
		
		\begin{claim}\label{claim3} There is $\alpha\in(0,1)$ such that for every $\lbb\in(0,1)$ there exists $C_\lbb\in(0,\9)$ such that
			\begin{equation}\label{e6.22}
			\int_{Q_T}\vp|\Delta\vf^\lbb_\vp|^{1+\alpha}w_\delta\,dtdx\le C_\lbb,\ \ff\vp\in(0,1).\end{equation}
			Hence, $\{\vp^{\frac1{1+\alpha}}\,\Delta\vf^\lbb_\vp\mid\vp\in(0,1)\}$ is equi-integrable, hence weakly relatively compact in $L^{1}(Q_T;w_\delta\,dtdx)$. 
			Therefore, selecting another subsequence  if necessary, for every $n\in\mathbb{N}$ as $k\to\9$, we find
			\begin{equation}\label{e6.23}
			\vp_k\Delta\vf^{\lbb_n}_{\vp_k}\to0\mbox{ weakly both in }L^1(Q_T;w_\delta\,dtdx)\mbox{ and }L^1(Q_T;\eta w_\delta\,dtdx).\end{equation}\end{claim}
		
		\noindent{\it Proof of Claim} \ref{claim3}. By the de la   Vall\'ee Poussin theorem and a diagonal argument, \eqref{e6.23} follows from \eqref{e6.22},   so  we only have to prove \eqref{e6.22}. To this purpose, fix
		\begin{equation}\label{e6.24}
		s\in\(\frac{2d(p-1)}{p},d(p-1)\).\end{equation}
		We note that this interval is not empty and that its left boundary point is strictly bigger than $2$, since $d\ge2$, $p>2$. Then,
		$$\frac{p-2}{p-1}\ \frac s{s-2}<d,$$and so we may choose $\alpha\in\(0,\frac13\)$ such that
		\begin{equation}\label{e6.25}
		0<\frac{p-2}{p-1}\ \frac{(3\alpha+1)s}{s(1-\alpha)-2(\alpha+1)}<d\end{equation}
		and also that
		\begin{equation}\label{e6.26}
		\alpha<\frac{d(p-1)}{(d+2)(p-2)+d}.\end{equation}
		Then, we multiply \eqref{e6.15}  by $sign\ \Delta\vf^\lbb_\vp|\Delta\vf^\lbb_\vp|^\alpha w_\delta$ and integrate over  $Q_T$ to obtain after rearranging
		\begin{equation}\label{e6.27}
		\barr{l}
		\dd\int_{Q_T}(\vrho_\delta+\vp)|\Delta\vf^\lbb_\vp|^{1+\alpha}w_\delta\,dtdx
		\le
		\dd\int_{Q_T}|\nabla\vrho_\delta|\,
		|\nabla\vf^\lbb_\vp|\,|\Delta\vf^\lbb_\vp|^\alpha\,w_\delta\,dtdx\vspace*{1,5mm}\\
		\qquad+
		\dd\int_{Q_T}|(\vf^\lbb_\vp)_t|\,
		|\Delta\vf^\lbb_\vp|^\alpha\,w_\delta\,dtdx
		+\dd\int_{Q_T}|f\eta_\lbb+\eta'_\lbb\vf_\vp|\,
		|\Delta\vf^\lbb_\vp|^\alpha\, w_\delta\,dtdx.\earr\end{equation}
		We point out that by \eqref{e6.16} the right hand side  of \eqref{e6.27} is finite, hence so is its left hand side. The second term on the right hand side  of \eqref{e6.27} can be estimated by
		\begin{equation}\label{e6.28}
		|(\vf^\lbb_\vp)_t|^2_{L^2(Q_T)}+r\int_{Q_T}|\Delta\vf^\lbb_\vp|^{1+\alpha}\vrho_\delta\,w_\delta\,dtdx
		+C_r\int_{Q_T}w^{\frac2{1-\alpha}}_\delta\,\vrho^{-\frac{2\alpha}{1-\alpha}}_\delta\,dtdx.
		\end{equation}
		Likewise, we can estimate the last term in \eqref{e6.27} by
		\begin{equation}\label{e6.29}
		|f\eta_\lbb+\eta'_\lbb\vf_\vp|^2_{L^2(Q_T)}
		+r\int_{Q_T}|\Delta\vf^\lbb_\vp|^{1+\alpha}\,\vrho_\delta\,w_\delta\,dtdx
		+C_r\int_{Q_T}w_\delta^{\frac2{1-\alpha}}\,\vrho^{-\frac{2\alpha}{1-\alpha}}_\delta\,dtdx,
		\end{equation}
		where $r\in(0,1)$ and $C_r$ is a large enough constant 
		independent of $\vp,\lbb$. 
		We note that, for $\alpha\in\(0,\frac13\)$ satisfying \eqref{e6.26}, the definition of $\omega_\delta$ and $\vrho_\delta$ implies that the last terms 
		in \eqref{e6.28} and \eqref{e6.29} are finite. Furthermore, by \eqref{e6.19}   the respective first terms in \eqref{e6.28}, \eqref{e6.29} are uniformly bounded in $\vp\in(0,1)$. Finally, the first term on the right hand side  of \eqref{e6.27}   can be estimated by
		$$r\int_{Q_T}|\Delta\vf^\lbb_\vp|^{\alpha+1}\vrho_\delta\,w_\delta\,dtdx+
		C_r\int_{Q_T}\frac{|\nabla\vrho_\delta|^{\alpha+1}}{\vrho^\alpha_\delta}|\nabla\vf^\lbb_\vp|^{\alpha+1}\,w_\delta\,dx,$$for $r\in(0,1)$ and $C_r$ as above, where the second integral is up to a constant bounded by
		$$\int_{Q_T}|\nabla\vrho_\delta|^{\frac{2(\alpha+1)}{1-\alpha}}\,\vrho^{-\frac{3\alpha+1}{1-\alpha}}_\delta\,w^{\frac2{1-\alpha}}_\delta\,dtdx+\int_{Q_T}|\nabla\vf^\lbb_\vp|^2\vrho_\delta\,dtdx$$
		of which the second integral is uniformly bounded in $\lbb,\vp\in(0,1)$ by \eqref{e6.18}  and the first integral is up to a constant bounded by
		$$\int_{Q_T}|\nabla\vrho_\delta|^s\,dtdx+\int_{Q_T}w^{\frac{2s}{s(1-\alpha)-2(\alpha+1)}}_\delta\,
		\vrho^{-\frac{(3\alpha+1)s}{s(1-\alpha)-2(\alpha+1)}}_\delta\,dtdx,$$
		where by   \eqref{e6.24} and  \eqref{e6.25} 
		this quantity is finite. Choosing $r$ small enough, we hence get from \eqref{e6.27} and the nonnegative of $\vrho_\delta$ the estimate \eqref{e6.22}, and Claim \ref{claim3} follows.\hfill$\Box$\bigskip
		
		Now, we can  finish the proof of Theorem  \ref{t5.10}.
		We have by Claim \ref{claim2} and \eqref{e6.15}
		$$
		\barr{rcl}
		\dd\int_{Q_T}fv\,dtdx&=&\dd\lim_{n\to\9}\ \lim_{k\to\9}
		\int_{Q_T}(f\eta_{\lbb_n}+
		\eta'_{\lbb_n}\vf_{\vp_k})v\,dtdx
		\vspace*{1,5mm}\\
		&=&\dd\lim_{n\to\9}\ \lim_{k\to\9}\int_{Q_T}
		\(\frac d{dt}\ \vf^{\lbb_n}_{\vp_k}
		+{\rm div}((\vrho_\delta+\vp_k)\nabla\vf^{\lbb_n}_{\vp_k})\)v\,dtdx,\earr$$
		which, taking into account that $v=(1-\eta)w_\delta$,  by Claim \ref{claim3} is equal to
		$$\lim_{n\to\9}\ \lim_{k\to\9}\int_{Q_T}
		\(\frac \pp{\pp t}\ \vf^{\lbb_n}_{\vp_k}
		+{\rm div}\(\vrho_\delta\nabla\vf^{\lbb_n}_{\vp_k}\)\)v\,dtdx,$$
		which in turn equals zero by \eqref{e6.17}. Hence, $\int_{Q_T}vf\,dtdx=0$ and so, because $f\in C^\9_0(Q_T)$ was arbitrary, the assertion follows.
	\end{proof}

	\appendix
	\section{Details to proof of Theorem \ref{t4.2}} 
	
	Here we give more details on some assertions made in the proof of Theorem \ref{t4.2}. Below, we denote by $\lesssim$ an inequality where we supress a multiplicative constant on the right hand side which only depends on $d,p$ and $T$. First, by definition of $k = (p-2+\frac p d)^{-1}$, we have
	\begin{align*}
	-k(p-2)(1+ \frac{p}{d(p-1)}) > -1 \iff p-2 + \frac{p(p-2)}{d(p-1)} < p-2 + \frac p d  \iff p-2 < p-1.
	\end{align*}
	So, since $\frac{k}{d}\frac{p-2}{p-1} >0$, also $ -k(p-2)(1+\frac{p}{d(p-1)}) + \frac{k}{d}\frac{p-2}{p-1} >-1 $. 	
	Next, 
	$\nabla |\nabla w_0(t)|^{p-2} \in L^1_{\textup{loc}}([0,\infty);L^1_{\textup{loc}}(\R^d)$ follows from \eqref{e4.5}. 
	Indeed, $|x|^{-\frac 1{p-1}} \in L^1_{\textup{loc}}(\R^d)$ and both exponents of $t$  appearing on the right hand side of \eqref{e4.5} are greater than $-1$. For the first exponent, this was already shown above; for the second, we estimate ${\bf1}_{|x|\le\beta t^{\frac kd}}|x|$ by $\beta t^{\frac kd}$ and  have
	\begin{align*}
	& -k(p-2)(1+ \frac{p}{d(p-1)}) - \frac{kp}{d(p-1)}   +  \frac k d > -1  \iff p - 2 + \frac{p(p-2)}{d(p-1)} + \frac{p}{d(p-1)} - \frac 1 d < p-2 + \frac p d
	\\&  \frac{p(p-2) +p}{p-1} < p+1 \iff p < p+1.
	\end{align*}
	Finally, we give details regarding $\nabla(|\nabla w_0|^{p-2})w_0 \in L^1_{\textup{loc}}([0,\infty);L^1(\R^d))$. Let $T>0$. We have
	\begin{align}
\label{aux890}
	\int_0^T \int_{\R^d}& |\nabla(|\nabla w_0|^{p-2})| w_0 \,dx dt
	\\& \notag\le \int_0^T t^{-k(p-2)(1+ \frac{p}{d(p-1)})} \bigg(\int_{\{|x| \le \beta t^{\frac k d}\}}  |x|^{-\frac{1}{p-1}} w_0 dx    +     t^{-\frac{kp}{d(p-1)}}\int_{\{|x|\le
		\beta t^{\frac k d}\}}  |x|\,w_0\,dx\bigg) \, dt
	\\& \notag\le \int_0^T t^{-k(p-2)(1+ \frac{p}{d(p-1)})} \bigg(  t^{-k}\int_{\{|x| \le \beta t^{\frac k d}\}} |x|^{-\frac{1}{p-1}}\,dx +  t^{-\frac{kp}{d(p-1)}+\frac k d} \int_{\R^d}w_0  \,dx \bigg) dt ,
	\end{align}
	where for the second inequality we used $w_0(t,x) \lesssim t^{-k}$. Note
	\begin{align*}
	\int_{\{|x| \le  \beta t^{\frac k d}\}} |x|^{-\frac{1}{p-1}}\,dx  \leq C_t \int_0^{\beta t^{\frac k d}} r^{-\frac{1}{p-1} +d - 1} dr = C(T) t^{\frac{k(d(p-1)-1)}{d(p-1)}},
	\end{align*}
	where $C_t$ is a $t$-dependent constant related to the surface area of the $d$-dimensional ball with radius $\beta t^{\frac k d}$. Since $d$ is fixed, $C(T):= \sup\limits_{t\in (0,T)}C_t$ is finite. Since we find
	\begin{align*}
	&-k(p-2)(1+ \frac{p}{d(p-1)}) + \frac{k(d(p-1)-p)}{d(p-1)}-k > -1 \iff  p-2 + \frac{p(p-2)}{d(p-1)}  + \frac{1}{d(p-1)}  < p-2 + \frac p d
	\\& \frac{p(p-2)+1}{p-1}< p \iff 1-p < 0,
	\end{align*}
	the first integral term on the right hand side of \eqref{aux890} is hence finite. Regarding the second integral, we observe, similarly as above, 
	$$ -k(p-2)(1+ \frac{p}{d(p-1)})    -\frac{kp}{d(p-1)}+\frac k d>-1 \iff 1-p < 0,$$
	hence the finiteness of the second integral follows from $\int_{\R^d} w_0(t,x)dx = 1$ for all $t >0$.
	Consequently, the right hand side  of \eqref{aux890} is finite for every $T>0$ and we are done.
	
	\section{Proof of Remark \ref{r6.2}}

	Clearly, $C^\9_0([0,T)\times\rrd)$ is dense in $C^{\9}_{0,b}([0,T)\times\rrd)$ (defined in Remark \ref{r6.2}) with respect to  uniform convergence of all partial derivatives (including zero derivatives) on $[0,T)\times\rrd$. Let $\vf\in C^\9_{0,b}([0,T)\times\rrd)$ and $(\vf_k)_{k\in\mathbb{N}}\subseteq C^\9_0([0,T)\times\rrd)$, such that $\vf_k\to\vf$ as $k\to\9$  in the above sense. Then, in particular,
	\begin{itemize}
		\item[(i)] $\vf_k(0)\to\vf(0)$ uniformly on $\rrd$.
		\item[(ii)] $(\vf_k)_t\to\vf_t,\ \nabla\vf_k\to\nabla\vf,\ \Delta\vf_k\to\Delta\vf$ uniformly on $[0,T)\times\rrd,$ where we abbreviate $\frac d{dt}\,\vf$ by $\vf_t$.
	\end{itemize}Since $v\in L^1(Q_T)\cap L^\9(Q_T)$ and $v_0\in L^1(\rrd)$,  $\vrho_\delta\in L^\9$, $\nabla\vrho_\delta\in L^2(Q_T)$, (i)--(ii) yield
	$$0=\lim_{k}
	\left(\int_{Q_T}\!\!\!v((\vf_k)_t+{\rm div}(\vrho_\delta\nabla\vf_k))dtdx+\int_\rrd\!\!\!\vf_k(0)v_0dx\right)
	=\int_{Q_T}\!\!\!v(\vf_t+{\rm div}(\vrho_\delta\nabla\vf))dtdx+\int_\rrd\!\!\!\vf(0)v_0dx,
	$$ which proves  Remark \ref{r6.2}.
	
	\section{Details for \eqref{e6.17}}   	
	
	We know that $\vf^\lbb_\vp\in W^{2,1}_2(Q_T)$. It is standard that $C^\9_{0,b}([0,T)\times\rrd)$ (as introduced in Remark \ref{r6.2}) is dense in $W^{2,1}_2(Q_T)\cap\{g\in W^{2,1}_2(Q_T):g(T)=0\}$ with respect to  the usual norm
	$$(|g|^{2,1}_2)^2:=|g|^2_{L^2(Q_T)}+|g_t|^2_{L^2(Q_T)}+|\nabla g|^2_{L^2(Q_T)}+|\Delta g|^2_{L^2(Q_T)}.$$
	Now, let $(\vf_k)_{k\in\mathbb{N}}\subseteq C^\9_{0,b}([0,T)\times\rrd)$ such that  $\vf_k\stackrel k\to\vf^\lbb_\vp$ with respect to  this norm.	Since $v_0\equiv0$ and since $v\in L^2(Q_T)$, $\vrho_\delta\in L^\9(Q_T)$, $\nabla\vrho_\delta\in L^2(Q_T),$ $v\in L^\9(Q_T)$, and ${\rm div}(\vrho_\delta\nabla\vf^\lbb_\vp)=\vrho_\delta\Delta\vf^\lbb_\vp+\nabla\vrho_\delta\cdot\nabla\vf^\lbb_\vp,$ we deduce
	$$0=\lim_k\int_{Q_T}v((\vf_k)_t+{\rm div}(\vrho_\delta\nabla\vf_k))dtdx=\int_{Q_T}v((\vf^\lbb_\vp)_t+{\rm div}(\vrho_\delta\nabla\vf^\lbb_\vp))dtdx,$$which proves \eqref{e6.17}.

	\section{Linear, linearized and  nonlinear \FP\ equations} 

Here we recall of linear and nonlinear \FP\ equations and their standard notion of distributional solution. The linear \FP\ equation associated with measurable coefficients $a_{ij},b_i:(0,\9)\times\rrd\to\rr,$ $1\le i,j\le d$, is the second-order parabolic differential equation for measures
\begin{align}\label{eD.1} 
\frac d{dt}\,\mu_t=\frac\pp{\pp x_i}\ \frac\pp{\pp x_j}\,(a_{ij}(t,x)\mu_t)-\frac\pp{\pp x_i}\,(b_i(t,x)\mu_t),\ (t,x)\in(0,\9)\times\rrd.
\end{align}
The equation is linear, since in contrast to \eqref{e1.3} the coefficients do not depend on the solution $\mu_t$. Usually, an initial condition $\mu_0=\nu$ is imposed for solutions to \eqref{eD.1}, where $\nu\in\mathcal{M}^+_b$.

\begin{dfn}\label{dD.1} 
  A {\it distributional solution to \eqref{eD.1} with initial condition $\nu$}  is a weakly continuous curve $(\mu_t)_{t\ge0}$ of signed locally finite Borel measures on $\rrd$ such that
\begin{align}\label{eD.2} 
\int^T_0\int_\rrd|a_{ij}(t,x)|+|b_i(t,x)|d\mu_t(x)dt<\9,\ \ff T>0, 
\end{align}	
and
$$\int_\rrd\psi d\mu_t(x)-\int_\rrd\psi d\nu=\int^t_0\int_\rrd a_{ij}(s,x)\,\frac\pp{\pp x_i}\,\frac\pp{\pp x_j}\,\psi(x)+b_i(s,x)\,\frac\pp{\pp x_i}\,\psi(x)d\mu_s(x)ds,\ \ff t\ge0,$$
for all $\psi\in C^\9_0(\rrd)$. A distributional solution  is called {\it probability solution}, if $\nu$ and each $\mu_t,$ $t\ge0$, are probability measures. Instead of the initial time $0$, one may consider an initial time $s>0$. It is obvious how to generalize the definition in this regard. In this case, the initial condition is the pair $(s,\nu)$ and the solution is defined on $[s,\9)$.
\end{dfn} 

For the nonlinear FPE \eqref{e1.3}, the notion of distributional solution is similar.

\begin{dfn}\label{dD.2}
	A {\it solution to \eqref{e1.3} with initial condition $\nu$} is a weakly continuous curve $(\mu_t)_{t\ge0}$ of signed locally finite Borel measures on $\rrd$ such that $(t,x)\mapsto a_{ij}(t,x,\mu_t)$ and $(t,x)\mapsto b_i(t,x,\mu_t)$ are product measurable on $(0,\9)\times\rrd$,
\begin{align}\label{eD.3}  
\int^T_0\int_\rrd|a_{ij}(t,x,\mu_t)|+|b_i(t,x,\mu_t)|d\mu_t(x)dt<\9,\ \ff T>0, 
\end{align}	
and
$$\int_\rrd\psi d\mu_t(x)-\int_\rrd\psi d\nu=\int^t_0\int_\rrd a_{ij}(s,x,\mu_s)\,\frac\pp{\pp x_i}\,\frac\pp{\pp x_j}\, \psi(x)+b_i(s,x,\mu_s)\,\frac\pp{\pp x_i}\,\psi(x)d\mu_s(x)ds,\ \ff t\ge0,$$
for all $\psi\in C^\9_0(\rrd)$. The notion of probability solution and the extension to initial times $s>0$ is as in  the linear case.
\end{dfn}

\noindent{\bf Linearized equations.} For the nonlinear FPEs \eqref{e1.3}, a standard way to linearize them is as follows. Fix a solution $(\mu_t)_{t\ge0}$ in the measure-coefficient of the nonlinear coefficients, i.e. consider the {\it linear} coefficients
$$\tilde a_{ij}(t,x):= a_{ij}(t,x,\mu_t),\ \ \tilde b _i(t,x):= b_i(t,x,\mu_t),$$
and consider the linear FPE with coefficients $\tilde a_{ij},\tilde b_i$. Clearly, $(\mu_t)_{t\ge0}$ is a distributional solution to this linear FPE in the sense of Definition \ref{dD.1}. In general, there can be further solutions, which are not related to the fixed solution $(\mu_t)_{t\ge0}$.

This way, we obtain the linearized $p$-Laplace equation \eqref{e5.4} by fixing in the coefficients \eqref{e1.13'} the fundamental solution $w^0$ from \eqref{e1.18} and Definition \ref{d5.7} is exactly Definition \ref{dD.1} for the linear FPE~\eqref{e5.4}.

	\paragraph{Acknowledgement.} The second and third named authors are funded by the Deutsche Forschungsgemeinschaft (DFG, German Research
	Foundation) -– Project-ID 317210226 -– SFB 1283. During the preparation of this paper the last named author spent
	very pleasant and scientifically stimulating stays at the
	University of Minnesota and the University of Madeira. He
	would like to thank Nick Krylov and Jos\'e Lu\'\i s da Silva
	respectively for their warmhearted hospitality.

\end{document}